%% file: main.tex
\documentclass[11pt]{article}

\usepackage{filecontents}

\usepackage[sorting=nyt,backend=bibtex,style=numeric,firstinits=true,maxbibnames=9,maxcitenames=2]{biblatex} 

\DefineBibliographyStrings{english}{%
  backrefpage = {$\uparrow$},
  backrefpages = {$\uparrow$},
}

\bibliography{controlvariate}

\usepackage[english]{babel}
\usepackage{inputenc}
\usepackage[T1]{fontenc}
\usepackage[usenames,dvipsnames,svgnames,table]{xcolor}	
\usepackage{soul}							
\usepackage{paralist}							   
\usepackage[shortlabels]{enumitem}        

\usepackage{verbatim}                           
\usepackage{spverbatim} 
\usepackage{listings}                              

\usepackage{graphicx}                            
\usepackage[tight]{subfigure}                 
\usepackage{rotating}								
\usepackage[footnotesize]{caption}					
\usepackage{tabulary,tabularx,tabu}					
\usepackage[table]{xcolor}
\usepackage{adjustbox}			

\usepackage{mathtools}                         
\usepackage{amsfonts}                          
\usepackage{amssymb}                          
\usepackage{dsfont,mathrsfs}                 
\usepackage{bm}										
\usepackage{slantsc}		
\usepackage[makeroom]{cancel}			

\usepackage{titletoc}					
\usepackage[nottoc,notbib]{tocbibind}	
\usepackage{secdot}                   
\usepackage{appendix}					

\usepackage{amsthm}              
\usepackage{xspace}                
\usepackage{pdflscape}                 
\allowdisplaybreaks                         
\usepackage[shortcuts]{extdash}    
\usepackage{ifpdf}                         
\usepackage{ifthen}				
\usepackage{microtype}          
\usepackage{pdfpages}			

\usepackage[paper=a4paper,top=2cm,bottom=2.2cm,inner=2cm,outer=2cm]{geometry}

\usepackage{setspace}				
\usepackage[citecolor=PineGreen,colorlinks=true,linkcolor=RedViolet]{hyperref}

\usepackage[nottoc]{tocbibind}

\usepackage[nameinlink]{cleveref}		
\newcommand{\crefrangeconjunction}{--}

\newtheorem{dummy}{Dummy}[section]              
\newtheorem{proposition}[dummy]{Proposition}
\Crefname{proposition}{Proposition}{Propositions}

\Crefname{lemma}{Lemma}{Lemmas}
\newtheorem{theorem}[dummy]{Theorem}
\Crefname{theorem}{Theorem}{Theorems}
\newtheorem{corollary}[dummy]{Corollary}
\theoremstyle{definition}
\newtheorem{definition}[dummy]{Definition}
\newtheorem{remark}[dummy]{Remark}

\newtheorem{property}[dummy]{Property}

\newcolumntype{Y}{>{\centering\arraybackslash}X}
\newcolumntype{C}[1]{>{\centering\let\newline\\\arraybackslash\hspace{0pt}}p{#1}}

\include{commands}

\dottedcontents{section}[3em]{}{2em}{1pc}
\dottedcontents{subsection}[5em]{}{2.4em}{1pc}

\newcommand{\footnoteremember}[2]
{
   \newcounter{#1}\footnote{#2}\setcounter{#1}{\value{footnote}}
}
\newcommand{\footnoterecall}[1]
{
   \footnotemark[\value{#1}]
}

\graphicspath{{plots/}{logos/}}

\title{Efficient simulation of ruin probabilities when claims are mixtures of heavy and light tails}

\author{
    Hansj\"{o}rg Albrecher\footnoteremember{UNIL}{The Faculty of Business and Economics, University of Lausanne, Quartier UNIL-Chamberonne B\^{a}timent Extranef, 1015 Lausanne, Switzerland}\\
    \small \texttt{hansjoerg.albrecher@unil.ch}\\
    \and
    Martin Bladt\footnoterecall{UNIL}\\
    \small \texttt{martin.bladt@unil.ch}\\
    \and
    Eleni Vatamidou\footnoterecall{UNIL}\\
    \small \texttt{eleni.vatamidou@unil.ch}\\
}

\date{}

\begin{document}

\maketitle

\begin{abstract}
    We consider the classical Cram\'{e}r\-/Lundberg risk model with  claim sizes that are mixtures of phase\-/type and subexponential variables. Exploiting a specific geometric compound representation, we propose control variate techniques to efficiently simulate the ruin probability in this situation. The resulting estimators perform well for both small and large initial capital. We quantify the variance reduction as well as the efficiency gain of our method over another fast standard technique based on the classical Pollaczek\-/Khinchine formula. We provide a numerical example to illustrate the performance, and show that for more time\-/consuming conditional Monte Carlo techniques, the new series representation also does not compare unfavorably to the one based on the Pollaczek\-/Khinchine formula. \\
    
\end{abstract}


\section{Introduction}
The study of ruin probabilities for insurance risk models is a classical topic in applied probability, see e.g.\ \cite{rolski-SPIF}. Explicit formulas for ruin probabilities are available only in specific situations. One such instance is the classical Cram\'{e}r\-/Lundberg risk model when claim sizes are of phase\-/type, see e.g.\ \cite{asmussen-RP} for more details. However, the tail of such phase\-/type distributions is exponentially bounded \cite{neuts-MGSSM}, whereas insurance data often suggest heavy tails \cite{albrecher-RASA}. In the presence of heavy tails one then typically has to resort to approximations or simulations, and to achieve accuracy for either of the two can be challenging. While highly efficient simulation techniques for ruin probabilities for exponentially bounded claims are available for a long time already (e.g.\ using Lundberg conjugation \cite[Ch.XV]{asmussen-RP}), the field of efficient simulation for heavy tails has only advanced significantly in more recent years and is an active field of research (cf.\ \cite{asmussen2015error,ghamami2012improving,juneja2007estimating,nguyen2014new} and \cite{asmussen-SS} for an overview).

Among the many possible modelling approaches for insurance claim sizes, in this paper we will be interested in mixture models, where with a certain probability \perturbationparameter a new claim is of a heavy\-/tailed type and with probability $1-\perturbationparameter$ it is of a certain light\-/tailed type. Such a co\-/existence of heavy and light tails is very intuitive in practice, see e.g.\ \cite{lee2012modeling,tzougas2014optimal}. For small \perturbationparameter, \cite{vatamidou2013correctedrisk} used a perturbation approach to devise a numerical approximation scheme for the determination of ruin probabilities in the presence of heavy tails in the spirit of corrected phase\-/type approximations. Their approach relied on an alternative representation of the Pollaczek\-/Khinchine (PK) formula that converges more quickly as $\perturbationparameter \to 0$, see also \cite{geiger}. Inspired by this approach, in this paper we want to study the potential of such an alternative representation for general mixture models and not necessarily small \perturbationparameter. The focus here will be to see whether large claim approximations can be used more efficiently as control variates in a simulation procedure than for algorithms based on the classical PK formula. We will show both theoretically and in a numerical implementation that this is indeed the case. The results in principle apply to any situation where claim sizes are a mixture between a tractable light\-/tailed and a heavy\-/tailed distribution for which the convolution of the two can be calculated explicitly. Moreover, even if the latter convolution can not be evaluated explicitly, the series representation can be advantageous.

We will also study the performance of the alternative series representation for a conditional Monte Carlo method developed by Asmussen \& Kroese \cite{asmussen2006improved}. The latter can be applied to the PK formula and leads to a significant reduction of variance for the ruin probability estimator, but at a considerable  additional computational cost. It will turn out that for this case, our series representation has no significant advantage over the classical PK approach, but the performance is not worse either.

The rest of the paper is organised as follows. \Cref{Section: Model description} describes the risk model based on the mixture of light- and heavy\-/tailed claims and provides some preliminaries. In \Cref{Section: Control variate}, we then construct a new control variate estimator for the ruin probability based on subexponential properties, which can exploit the advantage of exact ruin probability formulas for the light\-/tailed component in the mixture. We provide error bounds, investigate the tail behaviour, and quantify the resulting variance reduction when using the control variates, as well as the advantage of our approach to the analogous one based on the PK formula. We also consider the introduction of this alternative series representation for a conditional Monte Carlo framework in the spirit of \cite{asmussen2006improved}. In \Cref{Section: Numerical results}, we then perform numerical experiments and analyse the results. Finally, we conclude in \Cref{Section: Conclusions}.

\section{Model description and preliminaries}\label{Section: Model description}
We start with a short description of phase\-/type and subexponential distributions in \Cref{Section: Phase-type distributions,secsub}, as they are building blocks for the risk model of this paper, which is introduced in Section \ref{sec:model}.

\subsection{Phase-type distributions}\label{Section: Phase-type distributions}
Consider a state space $E=\{1,2,\dots, p,p+1\}$ and a Markov jump process $\{X_t\}_{t\ge 0}$ evolving on $E$. Assume that the first $p$ states are transient and the last remaining state $p+1$ is absorbing. The intensity matrix of this process is given by
\begin{equation*}
 \bm{\Lambda} 
=
\begin{pmatrix}
\bm{T} & \bm{t} \\
\bm{0} & 0 
\end{pmatrix},
\end{equation*}
where $\bm{T}$ is a $p\times p$\-/dimensional sub\-/intensity matrix, and it consists of the jump rates between the transient states. The initial distribution of $\{X_t\}_{t\ge 0}$ on the transient states $1,\dots,p$ is defined by the vector $\bm{\pi}=(\pi_1,\dots,\pi_p)$ with $\pi_k={\mathbb P}(X_0=k)$ for $k=1,\dots,p$. Let $\bm{e}$ be a $p$-dimensional column vector of $1$'s and $\bm{t}=-\bm{T}\bm{e}$.  
A phase\-/type distribution is then defined as the absorption time $\tau$ of $X_t$, that is, 
$\tau:=\inf\{t>0|X_t=p+1\}$ 
follows a phase\-/type distribution with parameters $\bm{\pi},\bm{T}$.

Phase\-/type distributions are natural and tractable extensions of the exponential distribution (which is retrieved for $p=1$), in the sense that their density and distribution functions are explicitly given for $x>0$ by the formulae
\begin{equation*}
f(x)=\bm{\pi} \exp(\bm{T}x)\bm{t} \qquad \text{and} \qquad
F(x)=1-\bm{\pi}\exp(\bm{T}x)\bm{e},
\end{equation*}
where the exponential of a matrix $M$ is defined as 
$$\exp(\bm{M})=\sum_{n=0}^\infty \frac{\bm{M}^n}{n!}. $$
The class of phase\-/type distributions has various attractive properties (it is e.g.\ closed under mixing, convolutions, exceedances, ordering etc.) and for phase\-/type distributed insurance claims there exist explicit formulas for ruin probabilities in a number of models (see \cite[Ch.IX]{asmussen-RP} for details). In addition, the class is dense (in the sense of weak convergence) among all distributions on the positive real line, so that in principle one may approximate any distribution arbitrarily well with a phase\-/type distribution. However, by construction phase\-/type distributions have an exponentially bounded tail, which is often too restrictive in applications.

\subsection{Subexponential distributions}\label{secsub}
In many situations, distributions with a tail heavier than exponential are a better description of the data. Among these, an important subclass is the one of subexponential distributions \subexponentials, i.e.\ for any $n \in \naturals$, 
\begin{equation}\label{subexp}
    \overline{\examplecdfsymbol^{*n}}(\initialcapital) \sim n \exampleccdf,\quad \text{as}\;u\to\infty,
\end{equation}
where $\exampleccdf=1-\examplecdf$ is the tail of the underlying distribution function $F$, see e.g.\ \cite{teugels1975class}. This mathematical definition is built around the intuition that the tail behavior of sums of independent such random variables is determined by the largest among them. The class of subexponential distributions comprises most heavy-tailed distributions of practical interest (including the Lognormal, Pareto, and heavy-tailed Weibull distribution). However, for exact calculations this class is not amenable and one typically has to resort to simulation in order to determine quantities like ruin probabilities with subexponential claims, and the latter is known to be challenging due to the rare event character (cf.\  \cite[Ch.XV]{asmussen-RP}). In the sequel, we will need the following well\-/known asymptotic property of subexponential distributions (see e.g.\ \cite[Cor.3.18]{foss-IHSD} or \cite[Cor.X.1.11]{asmussen-RP}):

\begin{property}\label{Property: Convolution of a subexponential distribution with a distribution of a lighter tail}
  Let $\examplecdfsymbol \in \subexponentials$ and let $A$ be any distribution with a lighter tail, i.e.\ $\overline{A}(\initialcapital) = o\left(\exampleccdf\right)$. Then for the convolution $A * \examplecdfsymbol$ of $A$ and \examplecdfsymbol we have $A * \examplecdfsymbol \in \subexponentials$ and $\overline{\left(A * \examplecdfsymbol\right)}(\initialcapital) \sim \exampleccdf$.
\end{property}

\subsection{The model}\label{sec:model}

Consider the classical Cram\'{e}r\-/Lundberg risk model for the surplus process of an insurance portfolio. The premium inflow is assumed at a constant rate (w.l.o.g.\ 1 per unit time) and claims arrive according to a homogeneous Poisson process $\{\poissonprocessmix\}_{t \geq 0}$ with rate \arrivalrate. The claim sizes $\claimsizemix \equalindistribution \claimsizemixrv$ are i.i.d.\ with common distribution function \claimsizedistributionmixsymbol, and are independent of $\{\poissonprocessmix\}$. If \initialcapital is the initial capital, the surplus at time $t$ is then given by
\begin{equation*}
  \riskreservemix = \initialcapital +t-\sum_{k=1}^{\poissonprocessmix}\claimsizemix.
\end{equation*}

We also define the claim surplus process $\claimsurplusmix =\initialcapital - \riskreservemix$ and its maximum $\supremumsurplusmix = \sup_{0 \leq t < \infty} \claimsurplusmix$. The probability \ruinmix of ultimate ruin is then
\begin{equation}\label{Equation: Definition of the mix ruin probability}
  \ruinmix = \pr(\supremumsurplusmix > \initialcapital).
\end{equation}
In addition, we assume that the safety loading condition $\claimratemix = \arrivalrate \meanclaimsizemix <1$ holds and thus the well\-/known Pollaczek\-/Khinchine (PK) formula 
\begin{equation}\label{Equation: Pollaczek-Khinchine formula mixture model}
  1 - \ruinmix = (1-\claimratemix)\sum_{k=0}^\infty \claimratemix^k \excessclaimsizedistributionmixconvolution{k}
\end{equation}
can be used for the evaluation of the ruin probability. 
Here $\excessclaimsizedistributionmix = \int_{0}^\initialcapital \big( 1 - \claimsizedistributionmix \big) dx / \e \claimsizemixrv$ is the distribution function of the stationary excess claim size  \excessclaimsizemixrv, see e.g.\ \cite{asmussen-RP}.

In this paper, we assume that claim sizes are of a mixture type. Concretely, \claimsizemixrv is phase\-/type with probability $1-\perturbationparameter$ and heavy\-/tailed (subexponential) with probability \perturbationparameter, where $\perturbationparameter \in (0,1)$. The phase\-/type claim sizes $\phclaimsize \equalindistribution \phclaimsizerv$ and the subexponential claim sizes $\htclaimsize \equalindistribution \htclaimsizerv$ are both assumed to have finite means \meanphclaimsize and \meanhtclaimsize, respectively. 
Denote by \ltexcessclaimsizedistributionmix, \ltexcessphclaimsizedistribution, and \ltexcesshtclaimsizedistribution the Laplace transforms of the stationary excess claim sizes $\excessclaimsizemix \equalindistribution \excessclaimsizemixrv$, $\excessphclaimsize \equalindistribution \excessphclaimsizerv$, and $\excesshtclaimsize \equalindistribution \excesshtclaimsizerv$, respectively. Moreover, we set $\phclaimrate:=\arrivalrate \meanphclaimsize$ and $\htclaimrate := \arrivalrate \meanhtclaimsize$, which means that the phase\-/type and heavy\-/tailed claims are responsible for expected aggregate claim size $(1-\perturbationparameter)\phclaimrate$ and $\perturbationparameter\htclaimrate$ per unit time, respectively. The expected overall aggregate claim size is then given by $\claimratemix = (1-\perturbationparameter)\phclaimrate + \perturbationparameter\htclaimrate$. In terms of Laplace transforms, the Pollaczek\-/Khinchine formula can be written as
\begin{align}\label{Equation: Laplace transform of the Polllaczek-Khinchine formula mixture model}
    \e e^{-s \supremumsurplusmix} \
                = (1-\claimratemix)\sum_{k=0}^\infty\claimratemix^k \big( \ltexcessclaimsizedistributionmix \big)^k
                = \frac{1-\claimratemix}{1 - \claimratemix\;  \ltexcessclaimsizedistributionmix}
                = \frac{1-(1-\perturbationparameter)\phclaimrate-\perturbationparameter\htclaimrate}{1-(1-\perturbationparameter)\phclaimrate \ltexcessphclaimsizedistribution - \perturbationparameter\htclaimrate \ltexcesshtclaimsizedistribution}.
\end{align}

Using representation \eqref{Equation: Laplace transform of the Polllaczek-Khinchine formula mixture model}, it was shown in \cite{vatamidou2013correctedrisk} that \ruinmix can be expressed as a series expansion involving the ruin probability of a risk process with purely phase\-/type claim sizes (base model). One easy way to establish a phase\-/type base model is by simply considering that $\claimsizedistributionmix = (1-~\perturbationparameter)\phclaimsizedistribution + \perturbationparameter$, $x\geq 0$, i.e.\ discard all heavy\-/tailed claim sizes. This base model, for which the claim size distribution has an atom at zero, is equivalent to the compound Poisson risk model in which claims arrive at rate $(1-\perturbationparameter)\arrivalrate$ and follow the distribution of \phclaimsizerv. We denote by \supremumsurplusdiscardsymbol the supremum of its corresponding claim surplus process and we set $\discardclaimrate = (1-\perturbationparameter) \phclaimrate$. The PK formula for this base model takes the form
\begin{equation}\label{Equation: Laplace transform of the Polllaczek-Khinchine formula discard model}
  \e e^{-s \supremumsurplusdiscardsymbol} = \frac{1-\discardclaimrate}{1-\discardclaimrate \ltexcessphclaimsizedistribution}.
\end{equation}

We denote by \ruindiscard the phase\-/type approximation of \ruinmix that is obtained when we apply Laplace inversion to \eqref{Equation: Laplace transform of the Polllaczek-Khinchine formula discard model}. The following series expansion of \ruinmix for the general risk process was shown in \cite[Th.1]{vatamidou2013correctedrisk}. In order to keep this paper self\-/contained, we repeat the short proof here in the present notation. 

\begin{theorem}[\cite{vatamidou2013correctedrisk}]\label{Theorem: Discard expansion of the ruin probability}
  We have 
  \begin{align}\label{form1}
    \ruinmix   =   \frac{1-\claimratemix}{1-\discardclaimrate}\,\ruindiscard +\frac{1-\claimratemix}{1-\discardclaimrate}\, \sum_{k=1}^\infty \left(\frac{\perturbationparameter\htclaimrate}{1-\discardclaimrate}\right)^k \convolutiondiscard{k} ,
  \end{align}
  where $\convolutiondiscard{k} = \pr(\supremumsurplusdiscard[0]+\supremumsurplusdiscard[1]+\dots+\supremumsurplusdiscard+\excesshtclaimsize[1]+\dots+\excesshtclaimsize > \initialcapital)$ and $\supremumsurplusdiscard \equalindistribution \supremumsurplusdiscardsymbol$. This expansion converges for all values of \initialcapital.
\end{theorem}

\begin{proof}
It can easily be derived that $\excessclaimsizemixrv = \bernoullirv \excessphclaimsize + (1-\bernoullirv)\excesshtclaimsize$, where $\bernoullirv \sim Bernoulli \big( \discardclaimrate/(\discardclaimrate + \perturbationparameter \htclaimrate) \big)$. Therefore 
$\displaystyle
    \ltexcessclaimsizedistributionmix = \frac{\discardclaimrate}{\discardclaimrate + \perturbationparameter \htclaimrate}\ltexcessphclaimsizedistribution + \frac{\perturbationparameter\htclaimrate}{\discardclaimrate + \perturbationparameter \htclaimrate} \ltexcesshtclaimsizedistribution
$, and we find by virtue of the binomial identity
\begin{align*}
    \big( \ltexcessclaimsizedistributionmix \big)^\ell
        = \frac1{(\discardclaimrate + \perturbationparameter \htclaimrate)^\ell}
        \sum_{k=0}^\ell \binom{\ell}{k} (\discardclaimrate)^{\ell-k} \big( \ltexcessphclaimsizedistribution \big)^{\ell-k} (\perturbationparameter \htclaimrate)^k \big( \ltexcesshtclaimsizedistribution \big)^{k}.
\end{align*}
Combining \Cref{Equation: Laplace transform of the Polllaczek-Khinchine formula mixture model,Equation: Laplace transform of the Polllaczek-Khinchine formula discard model}, we get
    \begin{align*}
     \e e^{-s \supremumsurplusmix} =&\ 
      (1-\discardclaimrate - \perturbationparameter \htclaimrate)\sum_{\ell=0}^\infty\sum_{k=0}^\ell \binom{\ell}{k} (\discardclaimrate)^{\ell-k} \big( \ltexcessphclaimsizedistribution \big)^{\ell-k} (\perturbationparameter \htclaimrate)^k \big( \ltexcesshtclaimsizedistribution \big)^{k}\\
    =&\  (1-\discardclaimrate - \perturbationparameter \htclaimrate)\sum_{k=0}^\infty (\perturbationparameter \htclaimrate)^k \big( \ltexcesshtclaimsizedistribution \big)^{k} \sum_{\ell=k}^\infty \binom{\ell}{k} (\discardclaimrate)^{\ell-k} \big( \ltexcessphclaimsizedistribution \big)^{\ell-k} \\
    =&\  (1-\discardclaimrate - \perturbationparameter \htclaimrate)\sum_{k=0}^\infty (\perturbationparameter \htclaimrate)^k \big( \ltexcesshtclaimsizedistribution \big)^{k} \frac1{\big( 1- \discardclaimrate \ltexcessphclaimsizedistribution \big)^{k+1}} \\
 =&\  (1-\discardclaimrate - \perturbationparameter \htclaimrate)\sum_{k=0}^\infty (\perturbationparameter \htclaimrate)^k \big( \ltexcesshtclaimsizedistribution \big)^{k} \frac1{( 1- \discardclaimrate )^{k+1}} \Big( \e e^{-s \supremumsurplusdiscardsymbol} \Big)^{k+1} \\
    =& \frac{1-\claimratemix}{1-\discardclaimrate}\, \sum_{k=0}^\infty \left(\frac{\perturbationparameter\htclaimrate}{1-\discardclaimrate}\right)^k \big( \ltexcesshtclaimsizedistribution \big)^{k} \Big( \e e^{-s \supremumsurplusdiscardsymbol} \Big)^{k+1} .
    \end{align*}
    We obtain the provided series expansion for \ruinmix via Laplace inversion and using $\ruindiscard= \pr(\supremumsurplusdiscard[0] > \initialcapital)$. The convergence is granted by $\displaystyle \abs[{\e e^{-s \supremumsurplusdiscardsymbol}}] \leq 1$ and $\displaystyle \abs[ \ltexcesshtclaimsizedistribution] \leq 1$, while $\perturbationparameter\htclaimrate < 1-\discardclaimrate$ due to the stability condition $\claimratemix <1$.
\end{proof}

\Cref{Theorem: Discard expansion of the ruin probability} provides an alternative interpretation for \supremumsurplusmix, i.e.\ $\supremumsurplusmix \equalindistribution \sum_{k=0}^\geometric (\supremumsurplusdiscard + \excesshtclaimsize)$, where $\excesshtclaimsize[0]:=0$ and \geometric is a geometric random variable $\geometric \sim Geom\left( \frac{1-\claimratemix}{1-\discardclaimrate}\right)$. In general, the term corresponding to $k=0$ is explicit. Note that for various subexponential distributions associated with \excesshtclaimsize, the term corresponding to $k=1$ in \eqref{form1} can be also calculated explicitly, so that
\begin{align*}
    \ruinmix   =&  \underbrace{\frac{1-\claimratemix}{1-\discardclaimrate} \, \ruindiscard
    + \frac{1-\claimratemix}{1-\discardclaimrate} \, \frac{\perturbationparameter\htclaimrate}{1-\discardclaimrate} \, \pr(\supremumsurplusdiscard[0]+\supremumsurplusdiscard[1]+\excesshtclaimsize[1] > \initialcapital)}_{{explicit}} \\
    &+\frac{1-\claimratemix}{1-\discardclaimrate} \, \sum_{k=2}^\infty \left(\frac{\perturbationparameter\htclaimrate}{1-\discardclaimrate}\right)^k \convolutiondiscard{k}.
  \end{align*}
Thus, to approximate \ruinmix, we only need to have an estimate for 
  \begin{align}
    \remainingterms :&=   \frac{1-\claimratemix}{1-\discardclaimrate}\, \sum_{k=2}^\infty \left(\frac{\perturbationparameter\htclaimrate}{1-\discardclaimrate}\right)^k \convolutiondiscard{k}
    = \left(\frac{\perturbationparameter\htclaimrate}{1-\discardclaimrate}\right)^2 \e \convolutiondiscard{\geometric+2} \notag \\
    &= \left(\frac{\perturbationparameter\htclaimrate}{1-\discardclaimrate}\right)^2 \pr(\supremumsurplusdiscard[0]+\supremumsurplusdiscard[1]+\dots+\supremumsurplusdiscard[\geometric+2]+\excesshtclaimsize[1]+\dots+\excesshtclaimsize[\geometric+2] > \initialcapital) \label{Equation: New target probability},
  \end{align}
which we want to approximate by simulating the tail of
\begin{equation}\label{Equation: Target random variable}
    \newsupremum \equalindistribution \supremumsurplusdiscard[0]+\supremumsurplusdiscard[1]+\excesshtclaimsize[1]  +\sum_{k=2}^{\geometric+2} (\supremumsurplusdiscard + \excesshtclaimsize),
\end{equation}
with $\geometric \sim Geom\left( \frac{1 -\claimratemix}{1-\discardclaimrate}\right)$.

Using the above representation, we propose in \Cref{Section: Control variate} efficient variance reduction techniques for this simulation based on suitably chosen control variates.

\section{Control variate techniques}\label{Section: Control variate}
Let \unknownrv be the random variable we must simulate in order to calculate its expectation $ \remainingterms=\e \unknownrv$.
The idea of a control variate is to use another random variable \knownrv, which has a known expectation $\e \knownrv$ and is strongly correlated with \unknownrv. Thus, the deviation of the simulated from the exact value of \knownrv may be used for improving the simulation accuracy for \unknownrv. If $\big( \unknown, \known \big)$, $i=1,2,\dots,\simnumber$, are independent copies of $\big( \unknownrv, \knownrv \big)$, then an efficient control variate estimator is defined as
\begin{equation}\label{Equation: General estimator}
    \estimremainingterms{\simnumber}:= \empiricalmeanunknown + \empiricalcovariance \big( \empiricalmeanknown -  \e \knownrv \big),
\end{equation}
where
\begin{equation}\label{Equation: General expectations for the estimator}
    \empiricalmeanunknown = \frac{\sum_{i=1}^\simnumber \unknown}{\simnumber},\  
    \empiricalmeanknown = \frac{\sum_{i=1}^\simnumber \known}{\simnumber}, \  
    \empiricalcovariance = -\frac{\sum_{i=1}^\simnumber \big( \unknown - \empiricalmeanunknown \big) \big( \known - \empiricalmeanknown \big)}{\sum_{i=1}^\simnumber  \big( \known - \empiricalmeanknown \big)^2}.
\end{equation}
Note that this choice of \empiricalcovariance based on the empirical correlation of \unknownrv and \knownrv optimizes the variance gain, see e.g.\ \cite{albrecher-RASA,asmussen-SS}. 
We assume now that the distribution of \excesshtclaimsizerv belongs to the class of subexponential distributions {satisfying} \eqref{subexp}. The construction of the concrete \knownrv below is inspired by \Cref{Property: Convolution of a subexponential distribution with a distribution of a lighter tail} given in \Cref{secsub}. That is, for sufficiently large \initialcapital, only the maximum of the subexponential claims will substantially contribute to the probability in \eqref{Equation: New target probability}. 

\subsection{Max of heavy tails}
\label{Section: Approach 1 -- Max of heavy tails}
It is immediately obvious from \Cref{Equation: New target probability} that we may take
\begin{equation}\label{Equation: Definition for the unknonw random variable for the control variate}
    \unknownrv =   \left(\frac{\perturbationparameter\htclaimrate}{1-\discardclaimrate} \right)^2 \indicatorfunction[\{\newsupremum > \initialcapital \}],
\end{equation}
and this variable will have the desired mean $ \e \unknownrv=\remainingterms$. We also define, for fixed $n \in \naturals$, the random variable
\begin{equation}
    \newsupremumtruncate :=   \max \{ \excesshtclaimsize[1],\dots,\excesshtclaimsize[\geometric+2] \}   \indicatorfunction[\{\geometric +2\leq n  \}],
\end{equation}
which will serve as a component of the control variate of \unknownrv.

\begin{definition}\label{Definition: Corrected discard approximation}
For a fixed $n \in \naturals$, define the control variate
\begin{align}
    \truncknownrv =&  \left(\frac{\perturbationparameter\htclaimrate}{1-\discardclaimrate} \right)^2 \indicatorfunction[\{\newsupremumtruncate > \initialcapital \}].
     \label{Equation: Definition of the control variate}
\intertext{The $n$th order approximation $\approxremainingterms{n} = \e \truncknownrv $ of \remainingterms is then}
    \approxremainingterms{n} =&  \left( \frac{1-\claimratemix}{1-\discardclaimrate}\right) \sum_{k=2}^n \left(\frac{\perturbationparameter\htclaimrate}{1-\discardclaimrate}\right)^k \pr \big( \max \{ \excesshtclaimsize[1],\dots,\excesshtclaimsize[k] \} > \initialcapital \big).\label{Equation: Definition of the nth order approximation of the ruin probability}
\end{align}
\end{definition}

By construction, \approxremainingterms{n} underestimates \remainingterms. 
Next we collect some properties of this approximation.

\subsubsection{Properties of the approximation}\label{Section: Properties of the approximation}
The following lower and upper bounds for the  approximation error can be obtained.

\begin{proposition}[Error bounds]\label{Theorem: Error bounds for the corrected discard approximation}
  The error of the approximation \approxremainingterms{n}, $n \in \naturals$, is bounded from above and below as follows:
  \begin{align*}
    \left(\frac{\perturbationparameter\htclaimrate}{1-\discardclaimrate}\right)^{n+1}  \convolutiondiscard{1} \quad
    \leq  \remainingterms - \approxremainingterms{n} \quad
    \leq & \left(\frac{\perturbationparameter\htclaimrate}{1-\discardclaimrate}\right)^{n+1} \\
    &+ \left( 1-\frac{\perturbationparameter\htclaimrate}{1-\discardclaimrate}\right) \left(\frac{\perturbationparameter\htclaimrate}{1-\discardclaimrate} \excesshtclaimsizedistribution \right)^2 
    \frac{1- \left(\frac{\perturbationparameter\htclaimrate}{1-\discardclaimrate} \excesshtclaimsizedistribution \right)^{n-1}}{1- \frac{\perturbationparameter\htclaimrate}{1-\discardclaimrate}  \excesshtclaimsizedistribution}.
  \end{align*}
\end{proposition}
\begin{proof}
  For simplicity of notation, we set $p:= \frac{\perturbationparameter\htclaimrate}{1-\discardclaimrate}$. The error of the approximation is equal to
  \begin{align*}
      \remainingterms - \approxremainingterms{n}
      =&\ ( 1-p ) \sum_{k=2}^\infty p^n \convolutiondiscard{k}
      - (1 - p) \sum_{k=2}^n p^k \pr \big( \max \{ \excesshtclaimsize[1],\dots,\excesshtclaimsize[k] \} > \initialcapital \big)  \\
      =&\ ( 1-p ) \sum_{k=2}^n p^k \Big( \convolutiondiscard{k} -\pr \big( \max \{ \excesshtclaimsize[1],\dots,\excesshtclaimsize[k] \} > \initialcapital \big) \Big)
      + (1 - p) \sum_{k=n+1}^{\infty} p^k \convolutiondiscard{k}.
    \intertext{For the upper bound, we use $\pr \big( \max \{ \excesshtclaimsize[1],\dots,\excesshtclaimsize[k] \} > \initialcapital \big)
    = 1- \big( \excesshtclaimsizedistribution \big)^k$ and $\convolutiondiscard{k} \leq 1$ to obtain}
    \remainingterms - \approxremainingterms{n}
    \leq & ( 1-p ) \sum_{k=2}^n p^k \pr \big( \max \{ \excesshtclaimsize[1],\dots,\excesshtclaimsize[k] \} \leq \initialcapital \big) + (1 - p) \sum_{k=n+1}^{\infty} p^k
    = p^{n+1} + ( 1-p ) \sum_{k=2}^n \big( p \excesshtclaimsizedistribution \big)^k.
    \intertext{For the lower bound, we take $\convolutiondiscard{k} \geq \pr \big( \max \{ \excesshtclaimsize[1],\dots,\excesshtclaimsize[k] \} > \initialcapital \big) $ when $k \leq n$ and $\convolutiondiscard{k} \geq \convolutiondiscard{1}$ otherwise, to calculate}
    \remainingterms - \approxremainingterms{n}
    \geq & \;(1 - p) \sum_{k=n+1}^{\infty} p^k \convolutiondiscard{1} = p^{n+1} \convolutiondiscard{1},
  \end{align*}
  and the proof is complete.
\end{proof}

\begin{proposition}[Tail behaviour]\label{Theorem: Tail behaviour of the corrected discard approximation}
  For $\excesshtclaimsizerv \in \subexponentials$, the $n$th approximation 
  \begin{align*}
      \approxruinmix{n} &:= \frac{1-\claimratemix}{1-\discardclaimrate} \, \ruindiscard
    + \frac{1-\claimratemix}{1-\discardclaimrate} \,\frac{\perturbationparameter\htclaimrate}{1-\discardclaimrate} \, \pr(\supremumsurplusdiscard[0]+\supremumsurplusdiscard[1]+\excesshtclaimsize[1] > \initialcapital) + \approxremainingterms{n}
     \intertext{of the target ruin probability $\psi(u)$ has the following tail behaviour:}
    \approxruinmix{n} &\sim \frac{\perturbationparameter\htclaimrate}{1-\claimratemix} \left(1 - (n+1) \left( \frac{\perturbationparameter\htclaimrate}{1-\discardclaimrate} \right)^n + n \left( \frac{\perturbationparameter\htclaimrate}{1-\discardclaimrate} \right)^{n+1} \right) \excesshtclaimsizedistributioncomplement,\quad u\to\infty. \end{align*}
\end{proposition}
\begin{proof}
  The approximation \ruindiscard has a phase\-/type representation; therefore, it is of order $\displaystyle o\left(\excesshtclaimsizedistributioncomplement\right)$. The same holds for the tail of the distribution of $\supremumsurplusdiscard[0]+\supremumsurplusdiscard[1]$. Moreover, since $\excesshtclaimsizerv \in \subexponentials$, from \Cref{Property: Convolution of a subexponential distribution with a distribution of a lighter tail} we obtain $\pr(\supremumsurplusdiscard[0]+\supremumsurplusdiscard[1]+\excesshtclaimsize[1]> \initialcapital) \sim \excesshtclaimsizedistributioncomplement$. Finally, from $ \pr \big( \max \{ \excesshtclaimsize[1],\dots,\excesshtclaimsize[n] \} > \initialcapital \big) \leq  \pr(\excesshtclaimsize[1]+\dots+\excesshtclaimsize[n] > \initialcapital)$ and \eqref{subexp}, we deduce that $ \pr \big( \max \{ \excesshtclaimsize[1],\dots,\excesshtclaimsize[n] \} > \initialcapital \big) \sim n \excesshtclaimsizedistributioncomplement$, which leads to the following result by inserting these asymptotic estimates into \Cref{Definition: Corrected discard approximation}:
  \begin{align*}
      \approxruinmix{n} \sim &\left( 1-\frac{\perturbationparameter\htclaimrate}{1-\discardclaimrate}\right) \sum_{k=1}^n k \left( \frac{\perturbationparameter\htclaimrate}{1-\discardclaimrate}\right)^k \excesshtclaimsizedistributioncomplement
      = \frac{\frac{\perturbationparameter\htclaimrate}{1-\discardclaimrate} \left(1 - (n+1) \left( \frac{\perturbationparameter\htclaimrate}{1-\discardclaimrate} \right)^n + n \left( \frac{\perturbationparameter\htclaimrate}{1-\discardclaimrate} \right)^{n+1} \right)}{1-\frac{\perturbationparameter\htclaimrate}{1-\discardclaimrate}} \excesshtclaimsizedistributioncomplement \\
      =& \frac{\perturbationparameter\htclaimrate}{1-\claimratemix} \left(1 - (n+1) \left( \frac{\perturbationparameter\htclaimrate}{1-\discardclaimrate} \right)^n + n \left( \frac{\perturbationparameter\htclaimrate}{1-\discardclaimrate} \right)^{n+1} \right) \excesshtclaimsizedistributioncomplement.
  \end{align*}
\end{proof}

\Cref{Theorem: Tail behaviour of the corrected discard approximation} (in comparison with Theorem~5 in \cite{vatamidou2013correctedrisk}) shows that \approxruinmix{n} nearly captures the asymptotic behaviour of the exact ruin probability
\begin{equation}\label{heavyt}\psi(u)\sim \frac{\perturbationparameter\htclaimrate}{1-\claimratemix} \,\excesshtclaimsizedistributioncomplement, \end{equation}
being off by a factor $\left(1 - (n+1) \left( \frac{\perturbationparameter\htclaimrate}{1-\discardclaimrate} \right)^n + n \left( \frac{\perturbationparameter\htclaimrate}{1-\discardclaimrate} \right)^{n+1} \right) \allowbreak \in (0,1)$. As expected, the tail of \approxruinmix{n} underestimates the tail of \ruinmix.



\subsubsection{Variance reduction}\label{Section: Variance reduction}
We consider now the bivariate simulation of i.i.d.\ copies of the random variables \newsupremum and \newsupremumtruncate:
\begin{equation}
    \big( \simnewsupremum, \simnewsupremumtruncate{n} \big), \quad i=1,2,\dots,\simnumber.
\end{equation}
For each fixed $n \in \naturals$, the estimator \eqref{Equation: General estimator} takes the form
\begin{equation}\label{est1}
    \estimapproxremainingterms{n}{\simnumber}:= \empiricalmeanunknown + \empiricalcovariance \big( \empiricalmeanknown -  \approxremainingterms{n} \big).
\end{equation}
We can now establish our main result.

\begin{theorem}[Variance reduction]\label{Theorem: Variance reduction}
  For each fixed $n \in \naturals$, the variance of the estimator \eqref{est1}
  behaves asymptotically as
  \begin{equation}
      \var \big( \estimapproxremainingterms{n}{\simnumber} \big) \sim \left( \frac{\perturbationparameter\htclaimrate}{1-\discardclaimrate} \right)^{n+3}
      \frac{1+  n \left( \frac{1-\claimratemix}{1-\discardclaimrate} \right)}{\frac{1-\claimratemix}{1-\discardclaimrate}} \cdot \frac{\excesshtclaimsizedistributioncomplement}{\simnumber}, \qquad \text{as } \initialcapital \rightarrow \infty
  \end{equation}
  and satisfies
  \begin{equation}\label{varred}
      \frac{\var \big( \estimapproxremainingterms{n}{\simnumber} \big)}{\var \big( \empiricalmeanunknown \big)} \rightarrow \left( \frac{\perturbationparameter\htclaimrate}{1-\discardclaimrate} \right)^{n-1}
      \frac{1+  n \left( \frac{1-\claimratemix}{1-\discardclaimrate} \right)}{1+ \frac{1-\claimratemix}{1-\discardclaimrate}  }, \qquad \text{as } \initialcapital \rightarrow \infty.
  \end{equation}
\end{theorem}
\begin{proof}
  Since $\e \empiricalmeanunknown = \remainingterms$, we know from \cite{asmussen-SS} that the proposed estimator has variance
  \begin{equation}
      \frac{1-\big(\correlation\big)^2}{\simnumber}\, \var \unknownrv,
  \end{equation}
  with correlation coefficient $\correlation = \corr \big( \unknownrv, \truncknownrv \big)$. By the definition of \newsupremumtruncate, $\{ \newsupremumtruncate > \initialcapital \} \allowbreak \subseteq \{ \newsupremum > \initialcapital \}$ and consequently $\indicatorfunction[\{\newsupremum > \initialcapital \}]\cdot \indicatorfunction[\{\newsupremumtruncate > \initialcapital \}] = \indicatorfunction[\{\newsupremumtruncate > \initialcapital \}]$. We calculate,
  \begin{align*}
      \cov \big( \unknownrv, \truncknownrv \big)       &= \left( \frac{\perturbationparameter\htclaimrate}{1-\discardclaimrate} \right)^4 \cov \Big( \indicatorfunction[\{\newsupremum > \initialcapital \}]\cdot \indicatorfunction[\{\newsupremumtruncate > \initialcapital \}] \Big) \\
      &= \left( \frac{\perturbationparameter\htclaimrate}{1-\discardclaimrate} \right)^4 \left( \e \Big( \indicatorfunction[\{\newsupremum > \initialcapital \}] \indicatorfunction[\{\newsupremumtruncate > \initialcapital \}] \Big) - \e   \indicatorfunction[\{\newsupremum > \initialcapital \}] \e \indicatorfunction[\{\newsupremumtruncate > \initialcapital \}] \right) \\
      &= \left( \frac{\perturbationparameter\htclaimrate}{1-\discardclaimrate} \right)^4 \left( \pr \big( \newsupremumtruncate > \initialcapital \big) - \pr \big( \newsupremum > \initialcapital \big) \pr \big( \newsupremumtruncate > \initialcapital \big) \right)\\
      &= \left( \frac{\perturbationparameter\htclaimrate}{1-\discardclaimrate} \right)^4 \pr \big( \newsupremumtruncate > \initialcapital \big) \pr \big( \newsupremum \leq \initialcapital \big).
      \intertext{Similarly, we find}
      \var \big( \unknownrv \big) &= \left( \frac{\perturbationparameter\htclaimrate}{1-\discardclaimrate} \right)^4 \pr \big( \newsupremum > \initialcapital \big) \pr \big( \newsupremum \leq \initialcapital \big), \text{ and}\\
      \var \big( \truncknownrv \big) &= \left( \frac{\perturbationparameter\htclaimrate}{1-\discardclaimrate} \right)^4 \pr \big( \newsupremumtruncate > \initialcapital \big) \pr \big( \newsupremumtruncate \leq \initialcapital \big).
  \end{align*}
  Hence, it is immediate that
  \begin{equation}
      1- \big(\correlation\big)^2 = \frac{1 - \pr \big( \newsupremumtruncate > \initialcapital \big) / \pr \big( \newsupremum > \initialcapital \big)}{1 - \pr \big( \newsupremumtruncate > \initialcapital \big)}.
  \end{equation}
  Following \Cref{Theorem: Tail behaviour of the corrected discard approximation}, we calculate
    \begin{align*}
      \pr \big( \newsupremumtruncate > \initialcapital \big) \sim &\left( 1-\frac{\perturbationparameter\htclaimrate}{1-\discardclaimrate}\right) \sum_{k=2}^n k \left( \frac{\perturbationparameter\htclaimrate}{1-\discardclaimrate}\right)^{k-2} \excesshtclaimsizedistributioncomplement \\
      =& \frac{2-\frac{\perturbationparameter\htclaimrate}{1-\discardclaimrate} - (n+1) \left( \frac{\perturbationparameter\htclaimrate}{1-\discardclaimrate} \right)^{n-1} + n \left( \frac{\perturbationparameter\htclaimrate}{1-\discardclaimrate} \right)^{n} }{1-\frac{\perturbationparameter\htclaimrate}{1-\discardclaimrate}} \excesshtclaimsizedistributioncomplement
      \end{align*}
      and
       \begin{align*}
      \pr \big( \newsupremum > \initialcapital \big) \sim &\, \frac{2-\frac{\perturbationparameter\htclaimrate}{1-\discardclaimrate} }{1-\frac{\perturbationparameter\htclaimrate}{1-\discardclaimrate}} \excesshtclaimsizedistributioncomplement,
  \end{align*}
  as $\initialcapital \rightarrow \infty$. We finally obtain 
  \begin{gather*}
      \frac{\pr \big( \newsupremumtruncate > \initialcapital \big)}{\pr \big( \newsupremum > \initialcapital \big)}
      \rightarrow
      1 - \frac{(n+1) \left( \frac{\perturbationparameter\htclaimrate}{1-\discardclaimrate} \right)^{n-1} - n \left( \frac{\perturbationparameter\htclaimrate}{1-\discardclaimrate} \right)^{n}}{2-\frac{\perturbationparameter\htclaimrate}{1-\discardclaimrate}},
      \intertext{so that}
      1- \big(\correlation\big)^2 \rightarrow
      \frac{(n+1) \left( \frac{\perturbationparameter\htclaimrate}{1-\discardclaimrate} \right)^{n-1} - n \left( \frac{\perturbationparameter\htclaimrate}{1-\discardclaimrate} \right)^{n}}{2-\frac{\perturbationparameter\htclaimrate}{1-\discardclaimrate}},
  \end{gather*}
  and the statement of the theorem follows.
\end{proof}

The above theorem quantifies the asymptotic variance reduction for fixed $n$ as \initialcapital increases, this reduction being arbitrarily large when $n$ is increased sufficiently. 

\subsection{Conditional Monte Carlo}\label{Section: Approach 2 -- Conditional Monte Carlo}
While the approach of \Cref{Section: Approach 1 -- Max of heavy tails} is the focus of this paper, for purposes of comparison and completeness we are also interested in the performance of the alternative series representation for the conditional Monte Carlo estimate and its variance reduction proposed in \cite{asmussen2006improved}. To that end, let us recap here its idea and present its application to our series representation. Define $\zerosummand = \supremumsurplusdiscard[0]$ and $\summand = \supremumsurplusdiscard +  \excesshtclaimsize$, $k=1,2,\dots$, so that $\newsupremum \equalindistribution \zerosummand + \sum_{k=1}^{\geometric+2} \summand$, where $\geometric \sim Geom\left( \frac{1- \claimratemix}{1-\discardclaimrate}\right)$ as before. 
\Cref{Equation: New target probability} can then be written as
\begin{align*}
    \remainingterms =&\left(\frac{\perturbationparameter\htclaimrate}{1-\discardclaimrate}\right)^2 \pr ( \zerosummand +  \summand[1] + \dots + \summand[\geometric+2] > \initialcapital).
\intertext{Note that for fixed $k \geq 1$ and $\maxsummand := \max \{  \summand[1] , \dots, \summand\}$, we have}
    &\pr ( \zerosummand +  \summand[1] + \dots + \summand > \initialcapital)
    = k\, \pr ( \summation[k] > \initialcapital - \zerosummand,\summand  = \maxsummand)\\
    =\,& k\, \pr ( \summand > \maxsummand[k-1], \summand  >\initialcapital - \zerosummand- \summation[k-1])
    = k\, \e \summanddistributioncomplementsymbol \big( \maxsummand[k-1] \vee (\initialcapital  - \zerosummand - \summation[k-1]) \big),
\intertext{where \summanddistributioncomplementsymbol is the common c.c.d.f.\ of the \summand's and $\summation = \sum_{k=1}^\ell \summand$, $\summation[0]=0$. Consequently, 
the random variable}
    \akunknownrv = & \left(\frac{\perturbationparameter\htclaimrate}{1-\discardclaimrate} \right)^2 (\geometric+2) \summanddistributioncomplementsymbol \big( \maxsummand[\geometric + 1] \vee (\initialcapital  - \zerosummand - \summation[\geometric + 1]) \big),
\end{align*}
has the target probability \remainingterms as its expectation. Notice that this variable plays the same role as \unknownrv in the previous approach.

We can further introduce $\geometric \summanddistributioncomplementsymbol (\initialcapital)$ as a control variate for the number of summands (see e.g.\ \cite{ghamami2012improving}).

\begin{definition}\label{Definition: Conditional Monte Carlo estimator}
We use the control variate
\begin{align*}
    \akknownrv =& \left(\frac{\perturbationparameter\htclaimrate}{1-\discardclaimrate} \right)^2 (\geometric + 2)  \summanddistributioncomplementsymbol (\initialcapital ). 
\intertext{The resulting approximation $\akapproxremainingterms = \e \akknownrv $ of \remainingterms then is}
    \akapproxremainingterms :=&  \left(\frac{\perturbationparameter\htclaimrate}{1-\discardclaimrate} \right)^2  \Bigg( \frac{\perturbationparameter\htclaimrate}{1-\claimratemix} + 2 \Bigg)  \summanddistributioncomplementsymbol (\initialcapital ) .
\end{align*}
\end{definition}

This control variate leads to the following Asmussen\-/Kroese (AK)\-/type estimator:
\begin{equation}\label{xxf}
    \akapproxruinmix{\simnumber}:= \akempiricalmeanunknown + \akempiricalcovariance \big( \akempiricalmeanknown -  \akapproxremainingterms \big),
\end{equation} 
where \akempiricalmeanunknown, \akempiricalmeanknown, and \akempiricalcovariance are calculated via \eqref{Equation: General expectations for the estimator} using \akunknownrv and \akknownrv.

\begin{remark}
 An alternative approach is to set $\firstsummand = \supremumsurplusdiscard[0] + \supremumsurplusdiscard[1] +  \excesshtclaimsize[1]$ and $\summand = \supremumsurplusdiscard +  \excesshtclaimsize$, $k=2,3,\dots$ and write \Cref{Equation: New target probability} as
 \begin{align*}
    \remainingterms &=\left(\frac{\perturbationparameter\htclaimrate}{1-\discardclaimrate}\right)^2 \pr ( \firstsummand +  \summand[2] + \dots + \summand[\geometric+2] > \initialcapital).
\end{align*}
Observe that all the random variables on the right hand side of this  equation are heavy\-/tailed and independent, but not identically distributed. Thus, using the AK estimator for non i.i.d.\ random variables established in \cite{chan2011rare}, we could instead construct a control variate based on the conditional Monte Carlo estimator
\begin{align*}
    \estimator = & \left(\frac{\perturbationparameter\htclaimrate}{1-\discardclaimrate} \right)^2 \Bigg(
     \firstsummanddistributioncomplementsymbol \big( \maxzerosummand[1] \vee (\initialcapital - \summation[\geometric+1] + \summand[1]) \big)   + (\geometric + 1) \summanddistributioncomplementsymbol \big( \maxzerosummand[(\geometric+2)] \vee (\initialcapital - \summation[\geometric] - \zerosummand) \big) \Bigg),
\end{align*}
where $\maxzerosummand[1] = \max \{ \summand[2] , \dots, \summand[\geometric+2]\}$ and $\maxzerosummand = \max \{ \firstsummand ,  \summand[2] , \dots , \summand[k-1], \summand[k+1], \summand[\geometric+2]\}$.
\end{remark}

\subsection{Comparison with the Pollaczek-Khinchine expansion}\label{Section: Comparison with the Pollaczek-Khinchine expansion}
For reference and the purpose of comparison, we also consider the estimators analogous to the ones in  \Cref{Section: Approach 1 -- Max of heavy tails,Section: Approach 2 -- Conditional Monte Carlo} using the usual PK series expansion of the ruin probability in \eqref{Equation: Pollaczek-Khinchine formula mixture model}, which we rewrite as 
\begin{align*}
    \ruinmix   =&  \underbrace{(1-\claimratemix)\claimratemix \,\excessclaimsizedistributionmixcomplement }_{explicit} 
    +\underbrace{(1-\claimratemix)\sum_{k=2}^\infty \claimratemix^k \big( 1- \excessclaimsizedistributionmixconvolution{k} \big)}_{:=\pkremainingterms}.
  \end{align*}
Define the random variables $\pkgeometric \sim Geom(1-\claimratemix) $,
\begin{align*}
    \pknewsupremum &= \sum_{k=1}^{\pkgeometric+2} \excessclaimsizemix, \\
    \pknewsupremumtruncate &= \max \{ \excessclaimsizemix[1],\dots,\excessclaimsizemix[\pkgeometric+2] \}   \indicatorfunction[\{\pkgeometric \leq n -2 \}],
\end{align*}
and let $\pksummation = \sum_{k=1}^n \excessclaimsizemix $ as well as  $\pkmaxsummand = \max \{  \excessclaimsizemix[1] , \dots, \excessclaimsizemix\}$. With this notation, the following equations define the analogous control variate estimators of \pkremainingterms:
\begin{alignat*}{2}
    \pkunknownrv &=  \claimratemix^2 \indicatorfunction[\{\pknewsupremum > \initialcapital \}] \qquad \qquad
    \pkakunknownrv &&=  \claimratemix^2 (\pkgeometric+2) \excessclaimsizedistributionmixcomplement[ {\pkmaxsummand[\pkgeometric + 1] \vee (\initialcapital  - \pksummation[\pkgeometric + 1]) }]\\
    \pkknownrv &=  \claimratemix^2 \indicatorfunction[\{\pknewsupremumtruncate > \initialcapital \}] \qquad \qquad
    \pkakknownrv &&=  \claimratemix^2 (\pkgeometric + 2)  \excessclaimsizedistributionmixcomplement,
\end{alignat*}
and the associated empirical estimator
\begin{equation}\label{est1_pk}
    \estimapproxremainingtermsPK{n}{\simnumber}:=\empiricalmeanunknownPK+\empiricalcovariancePK\big(\empiricalmeanknownPK-\approxremainingtermsPK{n} \big).
\end{equation}

Observe now that the distributional behaviour of the variable \excessclaimsizemixrv is slightly different from that of \excesshtclaimsize. Recall that $\excessclaimsizemixrv = \bernoullirv \excessphclaimsize + (1-\bernoullirv)\excesshtclaimsize$, where $\bernoullirv \sim Bernoulli \big( \discardclaimrate/(\discardclaimrate + \perturbationparameter \htclaimrate) \big)$. Hence,
\begin{align*}
    \pr(\excessclaimsizemixrv > \initialcapital) &=\frac{\discardclaimrate}{\discardclaimrate + \perturbationparameter \htclaimrate} \,\pr(\excessphclaimsizerv > \initialcapital)+\frac{\perturbationparameter\htclaimrate}{\discardclaimrate + \perturbationparameter \htclaimrate}\, \pr(\excesshtclaimsizerv > \initialcapital)\sim \frac{\perturbationparameter\htclaimrate}{\discardclaimrate + \perturbationparameter \htclaimrate} \,\pr(\excesshtclaimsizerv > \initialcapital),
\end{align*}
as $\initialcapital \to \infty$. Moreover, since \excesshtclaimsizerv is subexponential, the above relation implies that \excessclaimsizemixrv is subexponential as well. Consequently, 
\begin{align*}
    \pr(\pkmaxsummand> \initialcapital)
    \sim k \,\frac{\perturbationparameter\htclaimrate}{\discardclaimrate + \perturbationparameter \htclaimrate}\,\pr(\excesshtclaimsizerv > \initialcapital)
    = k \,\frac{\perturbationparameter\htclaimrate}{\discardclaimrate + \perturbationparameter \htclaimrate}\, \excesshtclaimsizedistributioncomplement.
\end{align*}

\noindent Using the above asymptotic expression and following the proof of \Cref{Theorem: Variance reduction}, we obtain the next result.

\begin{theorem}\label{thm39}
  For each fixed $n \in \naturals$, the variance of the estimator \eqref{est1_pk}
  behaves asymptotically as
  \begin{equation*}
      \var \big( \estimapproxremainingtermsPK{n}{\simnumber} \big) \sim  \claimratemix ^{n+3}
      \frac{1+  n \left( 1-\claimratemix \right)}{1-\claimratemix} \cdot \frac{\perturbationparameter \htclaimrate}{\claimratemix} \cdot \frac{\excesshtclaimsizedistributioncomplement}{\simnumber}, \qquad \text{as } \initialcapital \rightarrow \infty,
  \end{equation*}
  and satisfies
  \begin{equation}\label{varredpk}
      \frac{\var \big( \estimapproxremainingtermsPK{n}{\simnumber} \big)}{\var \big( \empiricalmeanunknownPK \big)} \rightarrow \claimratemix^{n-1}
      \frac{1+  n \left( 1-\claimratemix \right)}{1+(1-\claimratemix)  }, \qquad \text{as } \initialcapital \rightarrow \infty.
  \end{equation}
\end{theorem}
It follows that we can compare the asymptotic effect on the variance between the two different series expansions for the ruin probability, as well as the effect on the proportion of variance reduction due to the use of control variates:
\begin{corollary}\label{cor39}
  For each fixed $n \in \naturals$, the following relations hold:
  \begin{equation}\label{cor391}
      \frac{\var \big( \estimapproxremainingterms{n}{\simnumber} \big)}{\var \big( \estimapproxremainingtermsPK{n}{\simnumber} \big)} \sim 
      \bigg[ \frac{\perturbationparameter\htclaimrate}{1 - \discardclaimrate} \bigg/ \claimratemix \bigg]^{n+2}
      \frac{1+  n \left( \frac{1-\claimratemix}{1 - \discardclaimrate} \right) }{1+  n ( 1-\claimratemix)}, \qquad \text{as } \initialcapital \rightarrow \infty,
  \end{equation}
  and
  \begin{equation*}
      \left[\frac{\var \big( \estimapproxremainingterms{n}{\simnumber} \big)}{\var \big( \empiricalmeanunknown \big)}\right]\cdot\left[\frac{\var \big( \estimapproxremainingtermsPK{n}{\simnumber} \big)}{\var \big( \empiricalmeanunknownPK \big)}\right]^{-1} \rightarrow \bigg[ \frac{\perturbationparameter\htclaimrate}{1 - \discardclaimrate} \bigg/ \claimratemix \bigg]^{n-1}
      \frac{1+  n \left( \frac{1-\claimratemix}{1 - \discardclaimrate} \right)}{1+  n ( 1-\claimratemix )} \cdot \frac{1+ ( 1-\claimratemix )}{1+ \left( \frac{1-\claimratemix}{1 - \discardclaimrate} \right)}, \qquad \text{as } \initialcapital \rightarrow \infty.
  \end{equation*}
\end{corollary}

Notice that the inequality $\frac{\perturbationparameter\htclaimrate}{1 - \discardclaimrate} < \claimratemix$ is actually equivalent to the net profit condition $\claimratemix<1$. As a consequence, the terms involving powers of $\frac{\perturbationparameter\htclaimrate}{1 - \discardclaimrate} / \claimratemix < 1$ in the above result guarantee (for large $n$) a better performance of our new series representation over the classical Pollaczek\-/Khinchine expansion.

\begin{remark}
 Note that the quantity \discardclaimrate depends on \perturbationparameter and that  $\frac{\perturbationparameter\htclaimrate}{1 - \discardclaimrate}$ is increasing in \perturbationparameter for each fixed aggregate claim rate \claimratemix. Correspondingly, the smaller the proportion \perturbationparameter of heavy\-/tailed claims is, the more our new series representation outperforms the classical Pollaczek\-/Khinchine expansion. The latter is intuitive, since the largest term will then dominate the others even more strongly, making our approximation even more efficient. The above expressions allow to quantify this effect.
 
 
\end{remark}

\section{Numerical experiments}\label{Section: Numerical results}

In this section, we test and numerically illustrate the efficiency of our proposed technique, and compare it to the analogous classical simulation techniques based on the PK representation \eqref{Equation: Pollaczek-Khinchine formula mixture model} (see also  \cite[Ch.XV.2]{asmussen-RP}).

To perform our numerical experiments, we need to specify a mixture claim size distribution for which the distributions of $\supremumsurplusdiscard[0]+\supremumsurplusdiscard[1]+\excesshtclaimsize[1]$ and $\supremumsurplusdiscard[1]+\excesshtclaimsize[1]$ can be evaluated explicitly; note that the second convolution is only required for the AK estimator.

\subsection{Mixture of exponential and Pareto claim sizes}
For the phase\-/type claim sizes we choose an exponential distribution with rate \exponentialrate, i.e.\  $\phclaimsizedistributioncomplement =\excessphclaimsizedistributioncomplement = e^{-\exponentialrate \initialcapital}$, and $\meanphclaimsize = 1/\exponentialrate$. For the heavy\-/tailed claim sizes we consider a shifted Pareto distribution with shape parameter $\paretoshape>1$ and scale $\paretoscale>0$, i.e.\  $\displaystyle \htclaimsizedistributioncomplement = \left(1 + \initialcapital/\paretoscale \right)^{- \paretoshape}$ and  $\displaystyle \excesshtclaimsizedistributioncomplement = \left(1 + \initialcapital/\paretoscale \right)^{- (\paretoshape-1)}$, $u\ge 0$, with $\displaystyle \meanhtclaimsize = \paretoscale/(\paretoshape-1$).

The two tail probabilities of the aforementioned sums of variables are explicitly available. For instance, for $\mu=3$, $a=2$, $b=1$, $\epsilon=0.1$ and $\rho=0.99$ they are given by

\begin{align}
    \pr(\supremumsurplusdiscard[0]+\supremumsurplusdiscard[1]+\excesshtclaimsize[1] > \initialcapital) =& \frac{1}{25.600.000.000 (1+ \initialcapital)} \times \Bigg( -41200 (223427+264627 \initialcapital) \nonumber\\
    &+ 297 (1+\initialcapital) \bigg( 400 e^{-309 \initialcapital/400}(292973+91773 \initialcapital) \nonumber\\
    &+ 31827(1691 + 891 \initialcapital) e^{-309 (1+\initialcapital)/400} \left( \text{Ei} \Big( \frac{309(1+\initialcapital)}{400} \Big) - \text{Ei} \Big( \frac{309}{400} \Big) \right) \bigg) \Bigg)\nonumber\\
    \pr(\summand[1] > \initialcapital) =& \frac{103}{400(1+\initialcapital)} + \frac{297}{320000}\times \Bigg( 800 e^{-309\initialcapital/400}\nonumber \\
    &+618 e^{-309 (1+\initialcapital)/400} \left( \text{Ei} \Big( \frac{309(1+\initialcapital)}{400} \Big) - \text{Ei} \Big( \frac{309}{400} \Big) \right)\Bigg),\label{eix}
\end{align}
where $\text{Ei}(z) = -\int_{-z}^\infty \frac{e^{-t}}{t} dt$ is the exponential integral. For all other parameters that we consider, analogous formulas are used. Finally, we calculate
$
    \pr \big( \max \{ \excesshtclaimsize[1],\dots,\excesshtclaimsize[k] \} > \initialcapital \big)
    = 1- \big( 1- \left(1 + \initialcapital/\paretoscale \right)^{- (\paretoshape-1)} \big)^k
$.

\subsection{Parameters}
In all our experiments, we fixed $\exponentialrate = 3$ and $\paretoscale=1$, while we considered various combinations for the remaining parameters. Motivated by \cite{vatamidou2013correctedrisk}, we focused mainly on the cases $\claimratemix \in \{0.9, 0.99, 0.999\}$, where simulations involving heavy tails can be considerably problematic (known as the heavy\-/traffic regime in the related queueing context, cf.\ \cite{asm03}) and where the first two terms of \eqref{Equation: Definition of the nth order approximation of the ruin probability} are known to be unable to close the gap between the approximation and the exact ruin probability even for values of $\perturbationparameter=0.1$. For the remaining parameters we tested $\perturbationparameter \in \{ 0.1, 0.7\}$ and $\paretoshape \in \{ 2,3,4\}$.

\subsection{Results}

In all the presented examples, the order of \approxruinmix{n} is equal to $n=100$ and the number of simulations is $\simnumber = 10,000$.

We plot in \Cref{Figure.Error bounds} the simulated ruin probability that is obtained using the Monte Carlo estimator \eqref{Equation: Definition for the unknonw random variable for the control variate} together with the heavy\-/tail approximation \eqref{heavyt}. The dashed black lines depict the error bounds in \Cref{Theorem: Error bounds for the corrected discard approximation}. We observe in both graphs that the lower bound converges to the heavy\-/tail approximation \eqref{heavyt} as $\initialcapital \rightarrow \infty$. This behaviour is observed for any $n$ and is in accordance with theory. A similar statement holds for any \initialcapital as $n \rightarrow \infty$. Further empirical tests show that this convergence in $n$ is remarkably fast. However, one cannot draw a safe conclusion for which choice of parameters the lower bound is below or above the heavy\-/tail approximation. Finally, we observe in the left graph that the upper bound is not very tight, as expected by \Cref{Theorem: Error bounds for the corrected discard approximation}, since the chosen parameters give $\perturbationparameter\htclaimrate/(1-\discardclaimrate) = 0.875$. The bound becomes tighter in the right graph, where  $\perturbationparameter\htclaimrate/(1-\discardclaimrate) = 0.25$.

From this point on, let us fix the parameters to $\paretoshape=2$, $\perturbationparameter=0.1$, and $\claimratemix=0.99$ to allow for comparability between \Cref{Figure.MC,Figure.AK}. Moreover, we use a log\-/log scale. In \Cref{Figure.MC}, we plot the MC estimate \eqref{Equation: Definition for the unknonw random variable for the control variate}  (blue solid line) together with the control variate extension \eqref{est1} (black dashed line) against the heavy-tail approximation \eqref{heavyt}. We observe that the control variate technique outperforms the crude estimate 
\eqref{Equation: Definition for the unknonw random variable for the control variate} across the entire range of \initialcapital (see the variance plot on the right). \Cref{Figure.MC} also compares the simulation results with the ones based on the classical PK formula described in \Cref{Section: Comparison with the Pollaczek-Khinchine expansion}. For the crude version, the latter are competitive for large $u$, but perform worse for small \initialcapital. However, for the control variate, our new approach is always significantly and convincingly better. This nicely illustrates the theoretical asymptotic results of Section \ref{Section: Control variate}: note that for the present choice of parameters the control variate asymptotically reduces the variance by a factor 0.09 (the constant on the right-hand side of \eqref{varred}) for our series representation, to be compared with 0.73 for the analogous constant on the right-hand side of \eqref{varredpk} for the PK representation. Related to that, the constant on the right-hand side of \eqref{cor391} in \Cref{cor39} is 0.12, which means that our series representation reduces the asymptotic variance by almost 90\%, when control variates are used in both cases.  \\


In \Cref{Figure.AK}, we plot the simulated ruin probability with the AK estimator (blue solid line) and its control variate extension from \Cref{Section: Approach 2 -- Conditional Monte Carlo} as a function of the heavy\-/tail approximation \eqref{heavyt}. We consider in the plot both the PK and our new series expansion. One recognises that the asymptotic behaviour according to \eqref{heavyt} (red dotted line) is recovered for all four estimators for sufficiently large \initialcapital. The right graph illustrates that the introduction of the control variate is a significant improvement in terms of variance reduction for both the PK and our series, and that the two latter approaches perform similarly. The overall variance is much lower than for the method underlying  \Cref{Figure.MC}. However, one should keep in mind that in terms of computation time the AK estimator in \Cref{Figure.AK} is much more time\-/consuming (about 20--50 times in our implementations), as the integrals \eqref{eix} have to be evaluated \simnumber times, whereas for the method in \Cref{Figure.MC} only once for the explicit term in front.


For large $\rho$, the number of summands tends to be large, and the results of the presented simulations 
suggest that the approximation 
\begin{equation}\label{Equation: Asymptotic behaviour of the sum in our expansion}
    \pr(\supremumsurplusdiscard[0]+\supremumsurplusdiscard[1]+\dots+\supremumsurplusdiscard[k]+\excesshtclaimsize[1]+\dots+\excesshtclaimsize[k] > \initialcapital) \approx \pr \big( \max \{ \excesshtclaimsize[1],\dots,\excesshtclaimsize[k] \} > \initialcapital \big)
\end{equation}
is better than the one employed using the usual PK series expansion
\begin{equation}\label{Equation: Asymptotic behaviour of the sum in the PK}
    \pr(\excessclaimsizemix[1]+\dots+\excessclaimsizemix[k] > \initialcapital) \approx \pr \big( \max \{ \excessclaimsizemix[1],\dots,\excessclaimsizemix[k] \} > \initialcapital \big).
\end{equation}

Intuitively, the latter is comprised of mixtures of heavy\-/tailed and light\-/tailed variables, and hence the number of heavy tailed variables is thinned down, which is a drawback that our new method does not have. This is further supported by the plot in the left panel of \Cref{Figure.cor}, where the empirical correlations between the control variates are given. Concretely, when simulating from \eqref{Equation: Asymptotic behaviour of the sum in the PK}, only $100\cdot \frac{\perturbationparameter \htclaimrate}{\claimratemix}$\% of our \excessclaimsizemix's will actually be heavy\-/tailed and thus one loses too much information from the original presence of heavy\-/tailed \excesshtclaimsize's, in contrast to \eqref{Equation: Asymptotic behaviour of the sum in our expansion} where only the light tails are omitted and all heavy tails are kept. Consequently, the new control variate is much more efficient, cf.\ the factors $\perturbationparameter \htclaimrate/\claimratemix$ in \Cref{cor39}. In contrast, for the AK estimator the control variate does not significantly differ for the two  series representations, and therefore -- while the control variate itself is a huge improvement over the crude estimate (cf.\ \Cref{Figure.AK} (right)) -- 
there is no improvement from using the alternative representation. 

\begin{figure}[h]
        \center{ \includegraphics[width=0.48\textwidth]{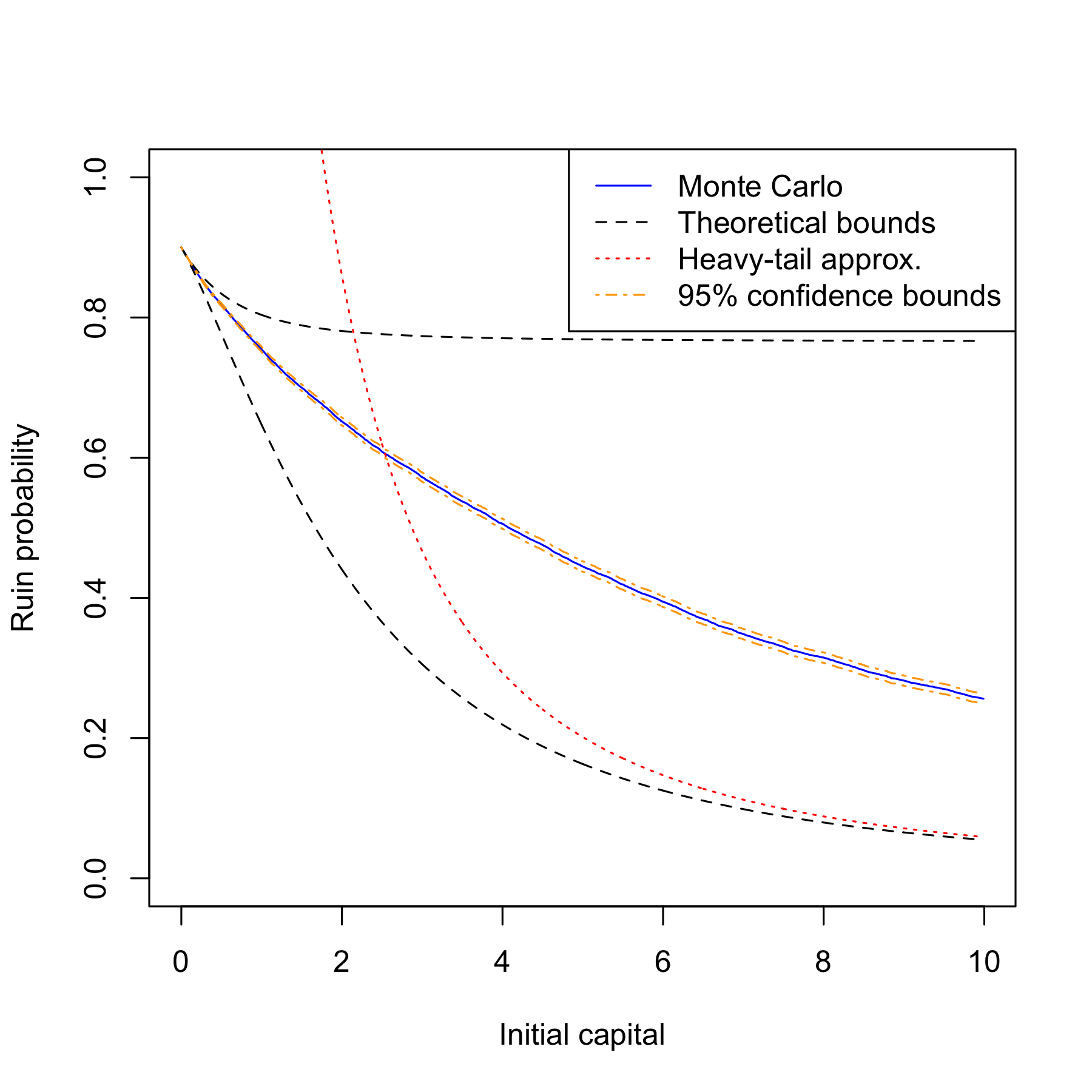} \hspace{0.2cm} \includegraphics[width=0.48\textwidth]{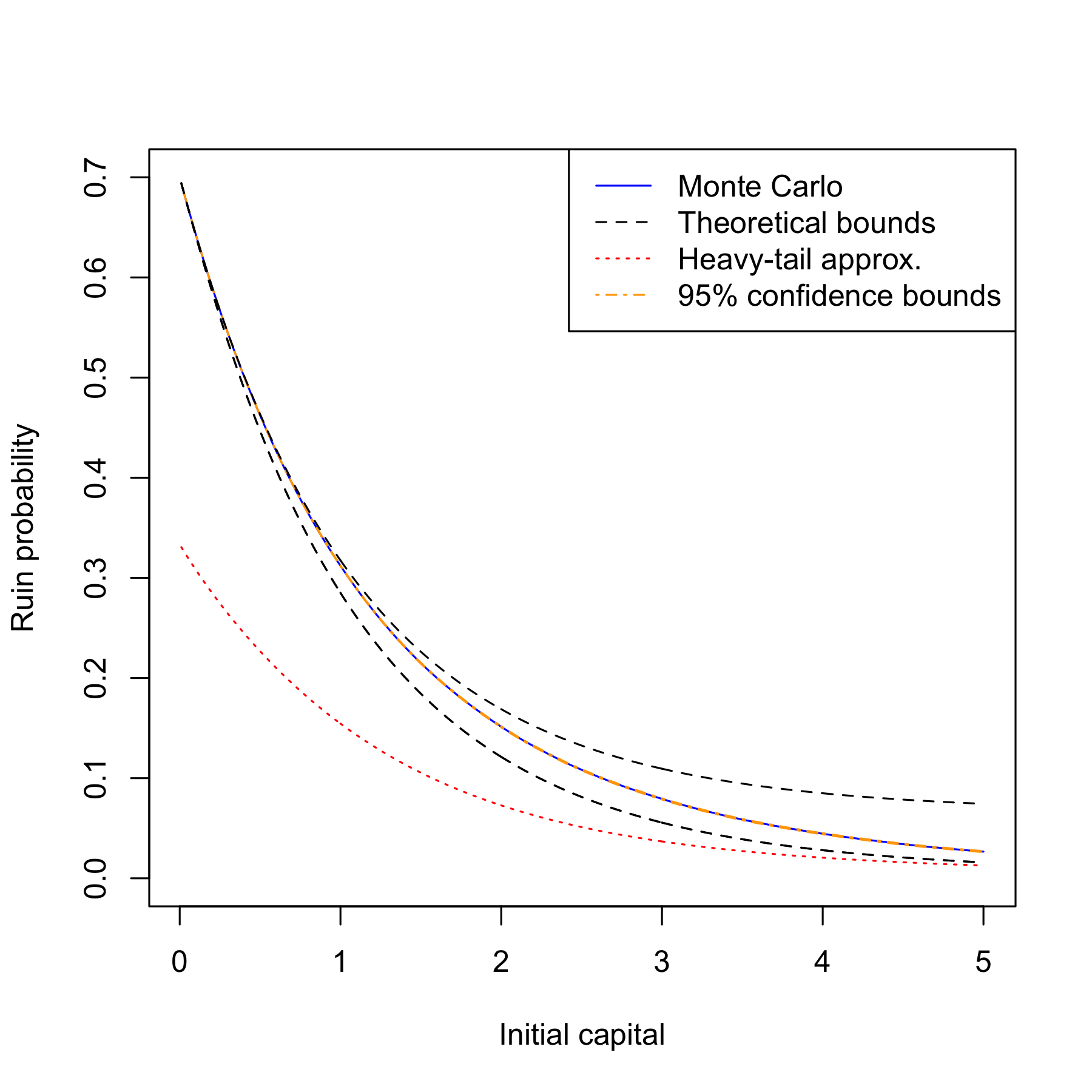}}
        \caption{The simulated ruin probability with MC estimator \eqref{Equation: Definition for the unknonw random variable for the control variate} (blue solid line together with its 95\% confidence interval in orange) and the heavy\-/tail approximation (red dotted line), as a function of the initial capital \initialcapital. The black dashed lines represent the error bounds in \Cref{Theorem: Error bounds for the corrected discard approximation}. Model parameters: $\paretoshape=3$ (both) and $\{\perturbationparameter, \claimratemix\}= \{0.7,0.9\}$ (left) or $\{\perturbationparameter, \claimratemix\}= \{0.1,0.7\}$ (right). }\label{Figure.Error bounds}
\end{figure}

\begin{figure}[h]
        \center{ \includegraphics[width=0.48\textwidth]{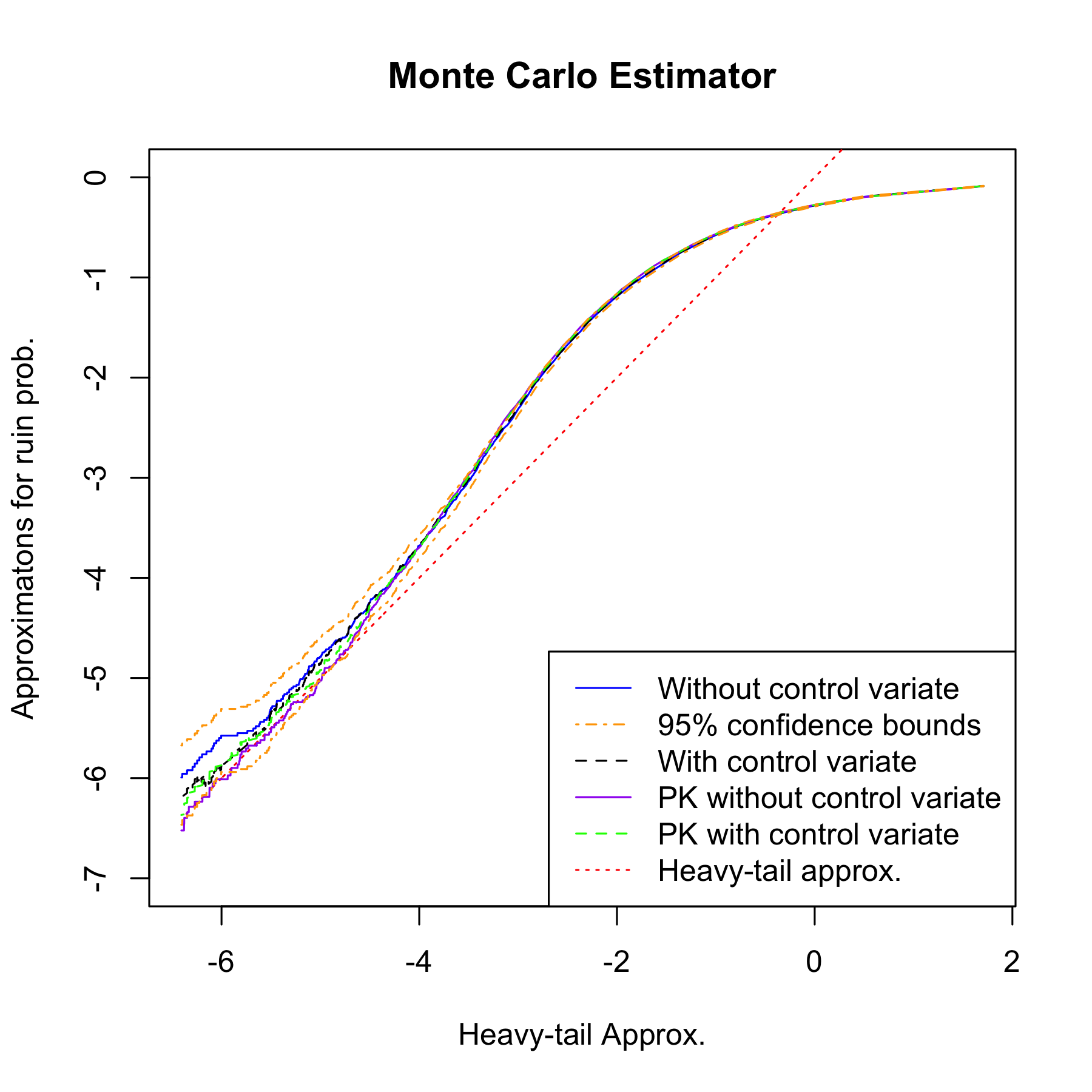} \hspace{0.2cm} \includegraphics[width=0.48\textwidth]{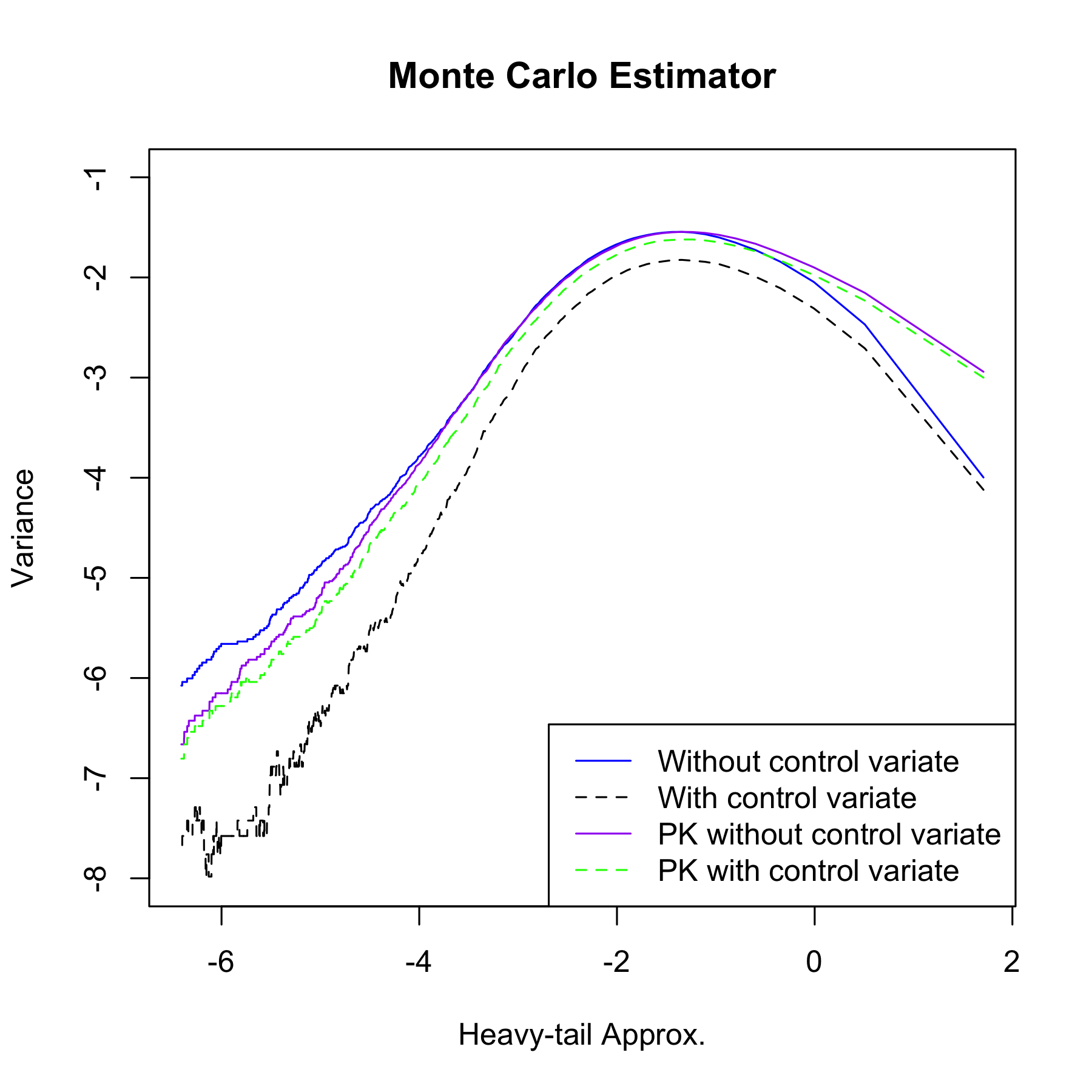}}
        \caption{The simulated ruin probability with MC estimator \eqref{Equation: Definition for the unknonw random variable for the control variate} (blue solid line together with its 95\% (blue solid line) and the control variate extension \eqref{est1} (black dashed line) against the heavy\-/tail approximation \eqref{heavyt}, plotted in a log-log scale. The corresponding estimates based on the PK formula are depicted in pink and dashed green. Model parameters: $\paretoshape=2$, $\perturbationparameter=0.1$, and $\claimratemix=0.99$. The respective empirical variances are presented on the right.}\label{Figure.MC}
\end{figure}

\begin{figure}[h]
        \center{ \includegraphics[width=0.48\textwidth]{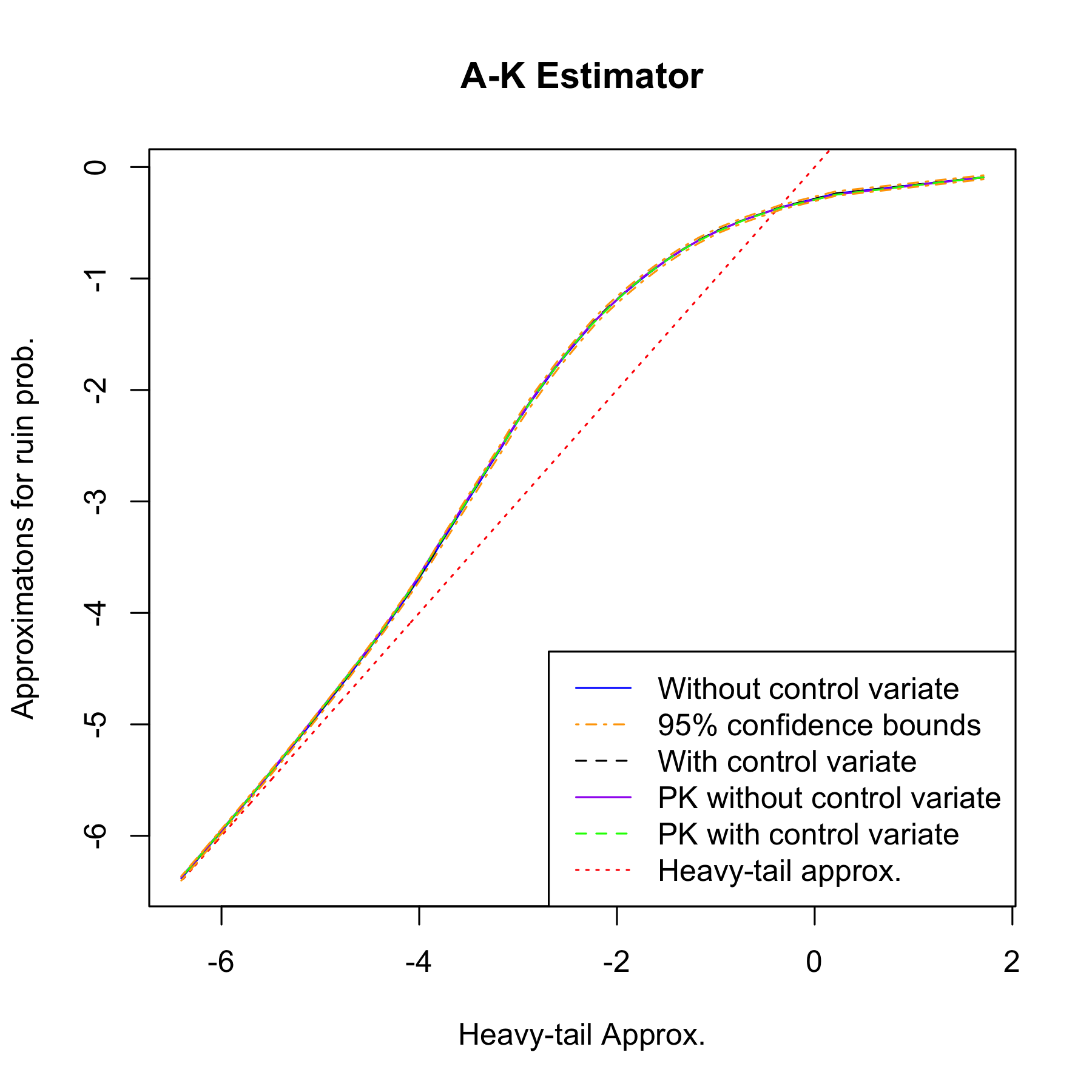} \hspace{0.2cm} \includegraphics[width=0.48\textwidth]{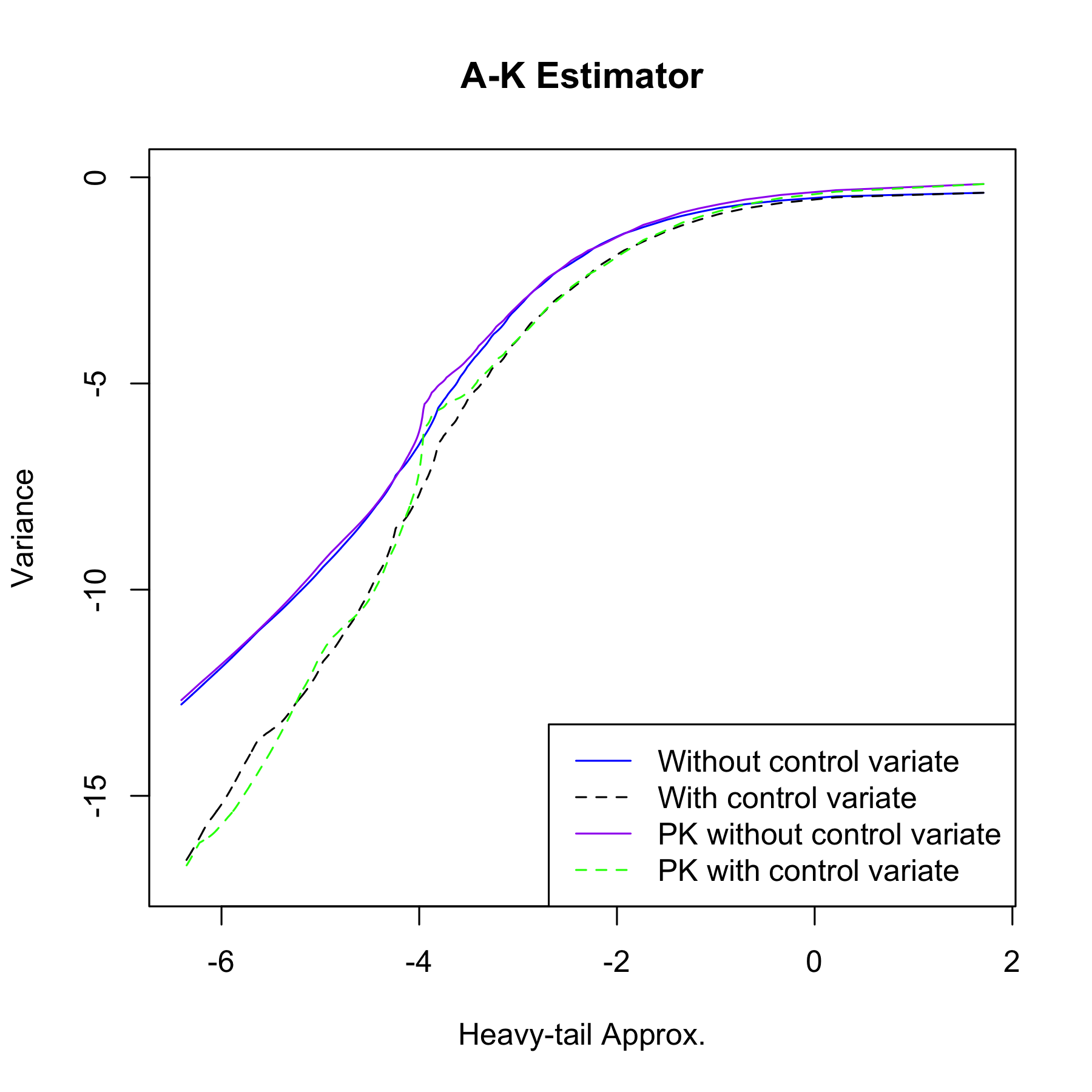}}
        \caption{The simulated ruin probability with the AK estimator (blue solid line) and its control variate extension  \eqref{xxf} (black dashed line) against the heavy\-/tail approximation \eqref{heavyt}, plotted in a log-log scale. The corresponding estimates based on the PK formula are depicted in pink and dashed green. Model parameters: $\paretoshape=2$, $\perturbationparameter=0.1$, and $\claimratemix=0.99$.  The respective empirical variances are presented on the right. }\label{Figure.AK}
\end{figure}

\begin{figure}[h]
        \center{ \includegraphics[width=0.48\textwidth]{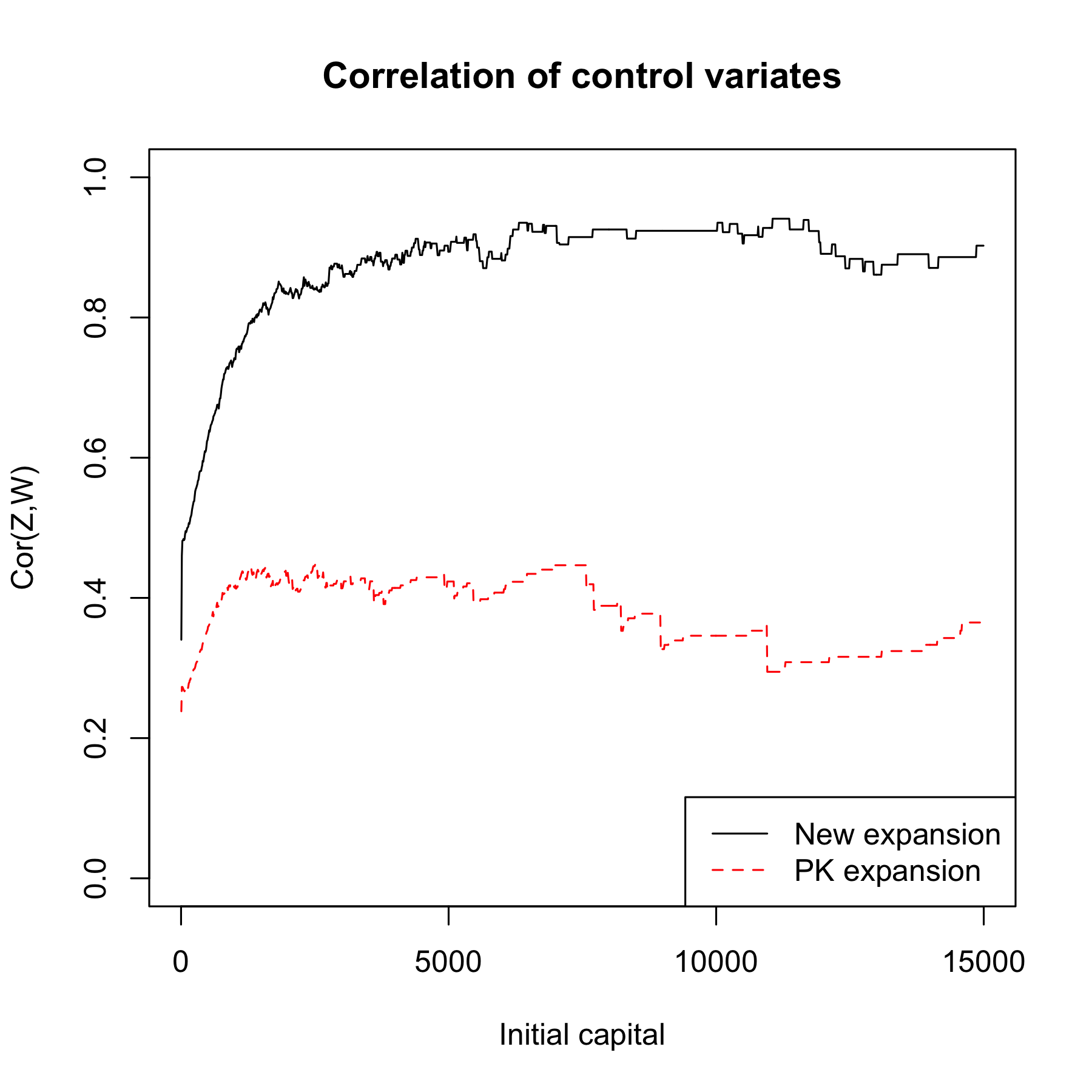} \hspace{0.2cm} \includegraphics[width=0.48\textwidth]{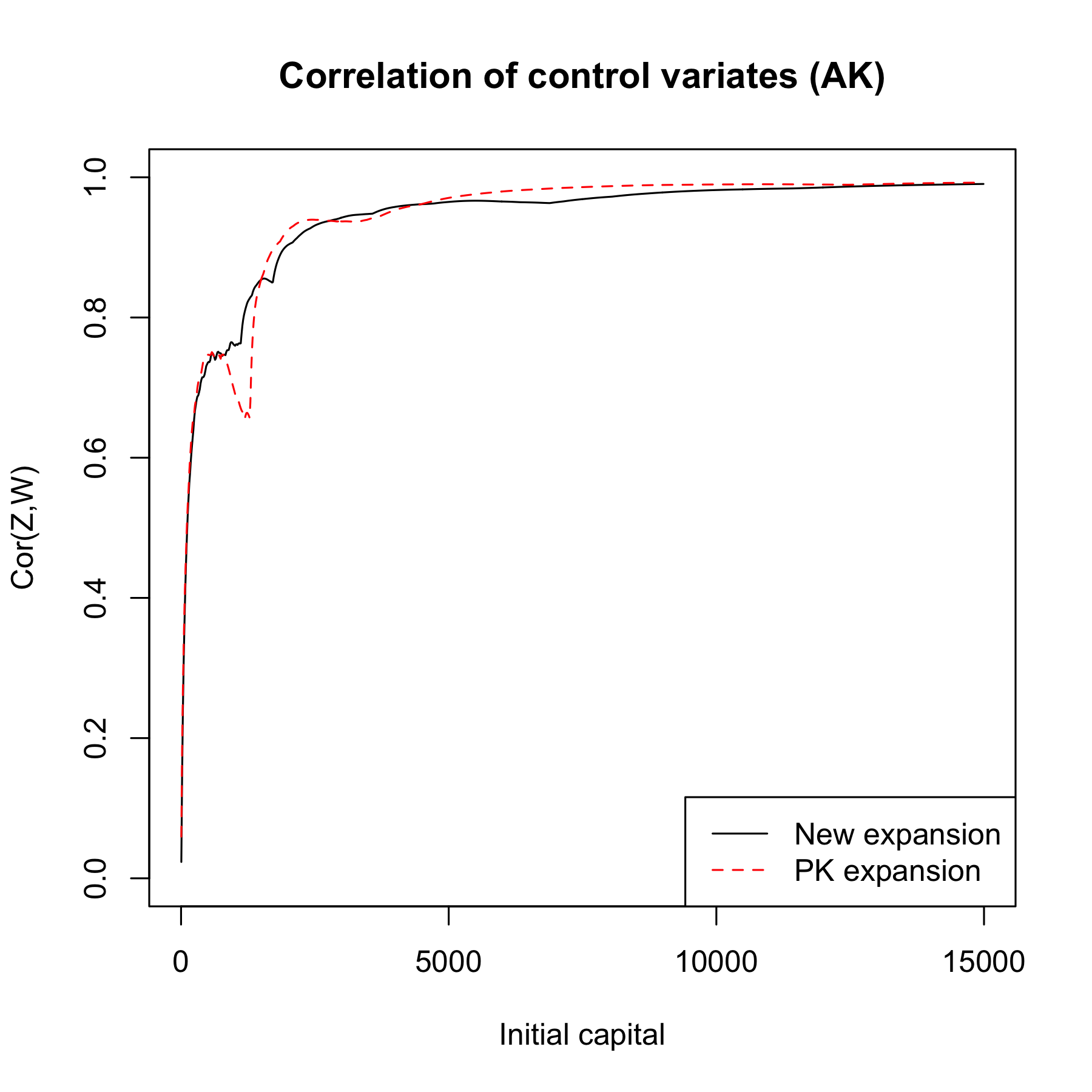}}
        \caption{The correlation of the control variate using the new series expansion (black solid line) and the classical one (red dashed line) for the traditional Monte Carlo approach (\Cref{Section: Variance reduction}, left panel) and for the AK approach (\Cref{Section: Approach 2 -- Conditional Monte Carlo}, right panel).  Model parameters: $\paretoshape=2$, $\perturbationparameter=0.1$, and $\claimratemix=0.99$. }\label{Figure.cor}
\end{figure}

\section{Conclusion}\label{Section: Conclusions}
In this paper, we introduced an alternative series expansion for the PK formula in the Cram\'{e}r\-/Lundberg model for the case when claims are mixtures of distributions with heavy and light tails. We showed that this can give rise to a significant improvement of simulation algorithms based on this series, both for large and small values of initial capital. 

When using the AK conditional Monte Carlo technique, the new series representation performs similarly as the original one based on the PK formula. Both these AK procedures (and particularly their control variate extensions w.r.t.\ \geometric) have a significantly lower variance for a fixed simulation size when compared to the method of \Cref{Section: Approach 1 -- Max of heavy tails}. However, the AK estimator is quite slow because it has to evaluate an improper integral in every iteration for the chosen mixture model. Hence, whenever time matters, the first simulation method based on \eqref{est1} is preferred, and there our new series is a significant improvement. The latter is particularly the case also in the heavy\-/traffic regime where simulation is known to be difficult. In addition, the performance is quite convincing also for moderate and low initial capital.

In addition, it is hard or even impossible to use the AK estimator when the distribution of $\supremumsurplusdiscard[1]+\excesshtclaimsize[1]$ is not known explicitly. On the other hand, our estimator can be used even if the probability $\pr(\supremumsurplusdiscard[0]+\supremumsurplusdiscard[1]+\excesshtclaimsize[1] > \initialcapital)$ cannot be calculated in a closed form. In such cases, one can simply simulate that latter probability as well and adapt the theoretical results in \Cref{Section: Properties of the approximation,Section: Variance reduction} accordingly.

In addition, although we concretely considered a mixture of a phase\-/type and a subexponential distribution in this paper, the results still hold if we replace \phclaimsizedistribution by any distribution for which $\ruindiscard = \pr(\supremumsurplusdiscard[0] > \initialcapital)$ has a closed form, e.g.\ matrix\-/exponential distributions (cf.\ \cite{bladt-MEDAP}). In addition, one can further modify our approach in order to evaluate \ruindiscard via simulation for any other light\-/tailed distribution, which is known to produce effortlessly reliable simulation outputs. 

Finally, we would like to point out that the ruin probability of the more general Sparre Andersen model also has a Pollaczek\-/Khinchine\-/type formula with respect to the ladder height distribution (\cite[Ch.VI]{asmussen-RP}). Our estimator is also valid for this model as long as the ladder height distribution can be found explicitly, which is for instance the case when the inter\-/occurrence times belong to the class of distributions with rational Laplace transform.

\section*{Acknowledgements}
H.A.\ and E.V.\ acknowledge financial support from the Swiss National Science Foundation Project 200021\_168993.


\phantomsection
\addcontentsline{toc}{section}{References}
\printbibliography

\end{document}

%% file: commands.tex
\newcommand{\e}[1][]{\ensuremath{\mathds{E}_{#1}}\xspace}    
\newcommand{\pr}[1][]{\ensuremath{\mathds{P}_{#1}}\xspace}   
\newcommand{\abs}[1][\cdot]{\ensuremath{\left| {#1} \right|}\xspace}        
\newcommand{\subexponentials}{\ensuremath{ \mathcal{S} }\xspace} 
\newcommand{\equalindistribution}{\ensuremath{\stackrel{\mathfrak{D}}{=}}\xspace}					
\newcommand{\indicatorfunction}[1][]{\ensuremath{\mathds{1}_{#1}}\xspace} 
\newcommand{\var}{\ensuremath{\mathrm{Var}}\xspace} 
\newcommand{\cov}{\ensuremath{\mathrm{Cov}}\xspace} 
\newcommand{\corr}{\ensuremath{\mathrm{Corr}}\xspace} 
\newcommand{\bernoullirv}{\ensuremath{I}\xspace} 

\newcommand{\naturals}[1][]{\ensuremath{\mathds{N}_{#1}}\xspace}        

\newcommand{\initialcapital}{\ensuremath{u}\xspace}        
\newcommand{\arrivalrate}{\ensuremath{\lambda}\xspace}            

\newcommand{\perturbationparameter}{\ensuremath{\epsilon}\xspace}                    

\newcommand{\riskreservemix}[1][t]{\ensuremath{R(#1)}\xspace}       			
\newcommand{\claimsurplusmix}[1][t]{\ensuremath{S(#1)}\xspace}       			
\newcommand{\poissonprocessmix}[1][t]{\ensuremath{N(#1)}\xspace}       		
\newcommand{\supremumsurplus}{\ensuremath{M}\xspace}           
\newcommand{\supremumsurplusmix}{\ensuremath{\supremumsurplus}\xspace}           
\newcommand{\newsupremum}{\ensuremath{V}\xspace}     
\newcommand{\newsupremumtruncate}[1][n]{\ensuremath{V_{#1}}\xspace}     
\newcommand{\simnewsupremum}[1][i]{\ensuremath{\newsupremum^{(#1)}}\xspace}           
\newcommand{\simnewsupremumtruncate}[2][i]{\ensuremath{\newsupremum^{(#1)}_{#2}}\xspace}           
\newcommand{\summand}[1][k]{\ensuremath{X_{#1}}\xspace}                     
\newcommand{\zerosummand}{\ensuremath{X_0^\star}\xspace}                     
\newcommand{\firstsummand}{\ensuremath{X_1^\star}\xspace}                     
\newcommand{\summation}[1][\ell]{\ensuremath{S_{#1}}\xspace}                     
\newcommand{\maxzerosummand}[1][k]{\ensuremath{m_{-#1}^\star}\xspace}
\newcommand{\maxsummand}[1][k]{\ensuremath{m_{#1}}\xspace}          
\newcommand{\estimator}{\ensuremath{Z^\star(\initialcapital)}\xspace}          
\newcommand{\firstsummanddistributioncomplementsymbol}{\ensuremath{\overline{F_{\firstsummand}}}\xspace}     
\newcommand{\summanddistributioncomplementsymbol}{\ensuremath{\overline{F}_{\summand[]}}\xspace}     

\newcommand{\excessclaimsizemixrv}{\ensuremath{\claimsizerv^e}\xspace}                     

\newcommand{\supremumsurplusdiscardsymbol}{\ensuremath{\supremumsurplus^\bullet}\xspace}           
\newcommand{\supremumsurplusdiscard}[1][k]{\ensuremath{\supremumsurplus^\bullet_{#1}}\xspace}           

\newcommand{\claimsizerv}{\ensuremath{U}\xspace}                     
\newcommand{\claimsizemixrv}{\ensuremath{\claimsizerv}\xspace}                     
\newcommand{\claimsizemix}[1][k]{\ensuremath{\claimsizerv_{#1}}\xspace}     
\newcommand{\claimsizedistributionsymbol}{\ensuremath{G}\xspace}     
\newcommand{\claimsizedistributionmixsymbol}{\ensuremath{\claimsizedistributionsymbol}\xspace}     
\newcommand{\meanclaimsizemix}{\ensuremath{\e\claimsizemixrv}\xspace}     
\newcommand{\phclaimsizedistributioncomplement}[1][\initialcapital]{\ensuremath{\overline{\phclaimsizedistributionsymbol}(#1)}\xspace}     
\newcommand{\excessclaimsizedistributionmixcomplement}[1][\initialcapital]{\ensuremath{\overline{\claimsizedistributionmixsymbol^e}(#1)}\xspace}     

\newcommand{\phclaimsizerv}{\ensuremath{B}\xspace}                                      
\newcommand{\phclaimsize}[1][k]{\ensuremath{\phclaimsizerv_{#1}}\xspace}     
\newcommand{\meanphclaimsize}{\ensuremath{\mu_B}\xspace}     
\newcommand{\excessphclaimsizerv}{\ensuremath{\phclaimsizerv^e}\xspace}        
\newcommand{\excessphclaimsize}[1][k]{\ensuremath{\excessphclaimsizerv_{#1}}\xspace}        
\newcommand{\phclaimsizedistributionsymbol}{\ensuremath{F_p}\xspace}     
\newcommand{\phclaimsizedistribution}[1][x]{\ensuremath{\phclaimsizedistributionsymbol(#1)}\xspace}     
\newcommand{\ltexcessphclaimsizedistribution}[1][s]{\ensuremath{\widetilde{F}^e_p(#1)}\xspace}     
\newcommand{\excessphclaimsizedistributioncomplement}[1][\initialcapital]{\ensuremath{\overline{\phclaimsizedistributionsymbol^e}(#1)}\xspace}     

\newcommand{\htclaimsizerv}{\ensuremath{C}\xspace}                                      
\newcommand{\htclaimsize}[1][k]{\ensuremath{\htclaimsizerv_{#1}}\xspace}                
\newcommand{\meanhtclaimsize}{\ensuremath{\mu_C}\xspace}     
\newcommand{\excesshtclaimsizerv}{\ensuremath{\htclaimsizerv^e}\xspace}        
\newcommand{\excesshtclaimsize}[1][k]{\ensuremath{\excesshtclaimsizerv_{#1}}\xspace}        
\newcommand{\htclaimsizedistributionsymbol}{\ensuremath{F_h}\xspace}     
\newcommand{\ltexcesshtclaimsizedistribution}[1][s]{\ensuremath{\widetilde{F}^e_h(#1)}\xspace}     
\newcommand{\htclaimsizedistributioncomplement}[1][\initialcapital]{\ensuremath{\overline{\htclaimsizedistributionsymbol}(#1)}\xspace}     
\newcommand{\excesshtclaimsizedistribution}[1][\initialcapital]{\ensuremath{\htclaimsizedistributionsymbol^e(#1)}\xspace}     
\newcommand{\excesshtclaimsizedistributioncomplement}[1][\initialcapital]{\ensuremath{\overline{\htclaimsizedistributionsymbol^e}(#1)}\xspace}     

\newcommand{\claimrate}{\ensuremath{\rho}\xspace}                    
\newcommand{\claimratemix}{\ensuremath{\claimrate}\xspace}                    
\newcommand{\discardclaimrate}{\ensuremath{\claimrate^\bullet}\xspace}                    
\newcommand{\phclaimrate}{\ensuremath{\delta}\xspace}                    
\newcommand{\htclaimrate}{\ensuremath{\theta}\xspace}                    

\newcommand{\ruinsymbol}{\ensuremath{\psi}\xspace}       
\newcommand{\ruinmix}[1][\initialcapital]{\ensuremath{\ruinsymbol(#1)}\xspace}                
\newcommand{\remainingterms}[1][\initialcapital]{\ensuremath{\varphi(#1)}\xspace}                
\newcommand{\pkremainingterms}[1][\initialcapital]{\ensuremath{\varphi^\circ(#1)}\xspace}                
\newcommand{\approxruinmix}[2][\initialcapital]{\ensuremath{\ruinsymbol_{#2}(#1)}\xspace}                
\newcommand{\approxremainingterms}[2][\initialcapital]{\ensuremath{\varphi_{#2}(#1)}\xspace}                
\newcommand{\approxremainingtermsPK}[2][\initialcapital]{\ensuremath{\varphi_{#2}^{\circ}(#1)}\xspace}                
\newcommand{\akapproxremainingterms}[1][\initialcapital]{\ensuremath{\varphi^{\star}(#1)}\xspace}     

\newcommand{\simnumber}{\ensuremath{\kappa}\xspace}  

\newcommand{\ruindiscard}[1][\initialcapital]{\ensuremath{\ruinsymbol^\bullet(#1)}\xspace}        

\newcommand{\excessclaimsizedistributionsymbol}{\ensuremath{\claimsizedistributionsymbol^e}\xspace}     
\newcommand{\excessclaimsizedistributionmixsymbol}{\ensuremath{\excessclaimsizedistributionsymbol}\xspace}     
\newcommand{\excessclaimsizedistributionmix}[1][\initialcapital]{\ensuremath{\excessclaimsizedistributionmixsymbol(#1)}\xspace}     
\newcommand{\excessclaimsizedistributionmixconvolution}[2][\initialcapital]{\ensuremath{(\excessclaimsizedistributionmixsymbol)^{* #2}(#1)}\xspace}     

\newcommand{\claimsizedistributionmix}[1][x]{\ensuremath{\claimsizedistributionmixsymbol(#1)}\xspace}     

\newcommand{\excessclaimsizemix}[1][k]{\ensuremath{\claimsizerv_{#1}^e}\xspace}     

\newcommand{\paretoshape}{\ensuremath{a}\xspace}    
\newcommand{\paretoscale}{\ensuremath{b}\xspace}    
\newcommand{\exponentialrate}{\ensuremath{\mu}\xspace}    

\newcommand{\ltexcessclaimsizedistributionmix}[1][s]{\ensuremath{\widetilde{\claimsizedistributionsymbol}^e(#1)}\xspace}     

\newcommand{\convolutiondiscard}[2][\initialcapital]{\ensuremath{\mathscr{A}_{#2}(#1)}\xspace}                      

\newcommand{\geometric}{\ensuremath{N}\xspace}  

\newcommand{\knownrv}{\ensuremath{W(\initialcapital)}\xspace}           
\newcommand{\truncknownrv}[1][n]{\ensuremath{W_{#1}(\initialcapital)}\xspace}           
\newcommand{\akknownrv}{\ensuremath{W^\star(\initialcapital)}\xspace}           
\newcommand{\known}[1][i]{\ensuremath{W^{(#1)}(\initialcapital)}\xspace}           
\newcommand{\unknownrv}{\ensuremath{Z(\initialcapital)}\xspace}           
\newcommand{\unknown}[1][i]{\ensuremath{Z^{(#1)}(\initialcapital)}\xspace}           
\newcommand{\akunknownrv}{\ensuremath{Z^\star(\initialcapital)}\xspace} 
\newcommand{\empiricalmeanknown}[1][\simnumber]{\ensuremath{\hat{w}_{#1}(\initialcapital)}\xspace}           
\newcommand{\empiricalmeanknownPK}[1][\simnumber]{\ensuremath{\hat{w}_{#1}^{\circ}(\initialcapital)}\xspace}           
\newcommand{\akempiricalmeanknown}[1][\simnumber]{\ensuremath{\hat{w}^\star_{#1}(\initialcapital)}\xspace}           
\newcommand{\empiricalmeanunknown}[1][\simnumber]{\ensuremath{\hat{z}_{#1}(\initialcapital)}\xspace}           
\newcommand{\empiricalmeanunknownPK}[1][\simnumber]{\ensuremath{\hat{z}_{#1}^{\circ}(\initialcapital)}\xspace}           
\newcommand{\akempiricalmeanunknown}[1][\simnumber]{\ensuremath{\hat{z}^\star_{#1}(\initialcapital)}\xspace}           
\newcommand{\empiricalcovariance}[1][\simnumber]{\ensuremath{\hat{\alpha}_{#1}}\xspace}           
\newcommand{\empiricalcovariancePK}[1][\simnumber]{\ensuremath{\hat{\alpha}_{#1}^{\circ}}\xspace}           
\newcommand{\akempiricalcovariance}[1][\simnumber]{\ensuremath{\hat{\alpha}^\star_{#1}}\xspace}           

\newcommand{\estimapproxremainingterms}[3][\initialcapital]{\ensuremath{\hat{\varphi}_{#3}^{#2}(#1)}\xspace}                
\newcommand{\estimapproxremainingtermsPK}[3][\initialcapital]{\ensuremath{\hat{\varphi}_{#3}^{\circ,#2}(#1)}\xspace}                
\newcommand{\estimremainingterms}[2][\initialcapital]{\ensuremath{\hat{\varphi}_{#2}(#1)}\xspace}   
\newcommand{\akapproxruinmix}[2][\initialcapital]{\ensuremath{\hat{\ruinsymbol}_{#2}^{\star}(#1)}\xspace}                

\newcommand{\correlation}[1][\initialcapital]{\ensuremath{r_n(#1)}\xspace}           

\newcommand{\examplecdfsymbol}{\ensuremath{F}\xspace}     
\newcommand{\examplecdf}[1][\initialcapital]{\ensuremath{\examplecdfsymbol(#1)}\xspace}     
\newcommand{\exampleccdf}[1][\initialcapital]{\ensuremath{\overline{\examplecdfsymbol}(#1)}\xspace}     

\newcommand{\pkknownrv}{\ensuremath{W^{\circ}_n(\initialcapital)}\xspace}           
\newcommand{\pkakknownrv}{\ensuremath{W^{\circ,\star}(\initialcapital)}\xspace}           
\newcommand{\pkunknownrv}{\ensuremath{Z^\circ(\initialcapital)}\xspace}   
\newcommand{\pkakunknownrv}{\ensuremath{Z^{\circ,\star}(\initialcapital)}\xspace}  
\newcommand{\pknewsupremum}{\ensuremath{V^\circ}\xspace}
\newcommand{\pknewsupremumtruncate}[1][n]{\ensuremath{V^\circ_{#1}}\xspace}  
\newcommand{\pkgeometric}{\ensuremath{N^\circ}\xspace}  
\newcommand{\pksummation}[1][n]{\ensuremath{S^\circ_{#1}}\xspace}           
\newcommand{\pkmaxsummand}[1][k]{\ensuremath{m^\circ_{#1}}\xspace}          
